\pdfoutput=1  

\documentclass{amsart}
\usepackage{amssymb}

\usepackage{graphicx}

\newtheorem{theorem}{Theorem}[section]
\newtheorem{lemma}[theorem]{Lemma}

\theoremstyle{definition}

\theoremstyle{remark}
\newtheorem{remark}[theorem]{Remark}

\numberwithin{equation}{section}

\begin{document}

\title[Quasi-alternating surgeries on asymmetric L--space knots]{Quasi-alternating surgeries on asymmetric L--space knots in the census}

\author[M. Teragaito]{Masakazu Teragaito}
\address{Department of Mathematics and Mathematics Education, Hiroshima University,
1-1-1 Kagamiyama, Higashi-hiroshima 7398524, Japan.}
\email{teragai@hiroshima-u.ac.jp}
\thanks{The author has been partially supported by JSPS KAKENHI Grant Number JP25K07004.}


\date{\today}



\begin{abstract}
In the SnapPy census, there are $9$ asymmetric L--space knots.
It is known that each of them admits  exactly two quasi-alternating surgeries with the aid of a computer.
The purpose of this article is to confirm these surgeries  by the Montesinos trick.
\end{abstract}


\maketitle


\section{Introduction}\label{sec:intro}

A knot $K$ in the $3$--sphere is called an \textit{L--space knot\/} if $K$ admits a non-trivial Dehn surgery yielding an L--space.
If the symmetry group of the knot complement is trivial, then a knot is said to be \textit{asymmetric}.
The first examples of asymmetric L--space knots were found by Baker and Luecke \cite{BL}.
Among them, the simplest one is the closure of a positive $12$--strand braid of length $249$, which has genus $119$,
as they mentioned in \cite[Example 1.7 and Figure 1]{BL}.

Among $1267$ hyperbolic knots in the SnapPy census, Dunfield \cite{D1,D2} gave $630$ L--space knots, leaving two as undetermined.
Later, these two are confirmed as L--space knots by  \cite{BKM}.
Thus there are exactly $632$ hyperbolic L--space knots in the census at the moment.
In \cite{ABG}, the computation of isometry groups by SnapPy \cite{CD} reveals that
there are exactly $9$ asymmetric L--space knots among the census knots:
\[
\begin{split}
& t12533, t12681, o9\_38928, o9\_39162, o9\_40363,   \\
& o9\_40487, o9\_40504, o9\_40582, o9\_42675.
\end{split}
\]
The simplest $t12533$ is the closure of a $4$--braid of length $27$, having genus $12$.
The remaining $623$ are strongly invertible.

Based on the SnapPy calculation,
Baker, Kegel and McCoy \cite[Theorem 4.2]{BKM}  verify that
each of the above asymmetric  L--space knots admits exactly two quasi-alternating surgeries,
which yield the double branched covers of quasi-alternating knots or links.
Table \ref{table:QA} shows the list of these surgeries quoted from \cite{BKM}.
All branching sets are known to be quasi-alternating, but it is not hard to verify the fact directly.

\begin{table}[ht]
 \caption{Quasi-alternating surgeries of asymmetric census L--space knots.}
\begin{tabular}{lll | lll}
\hline
knot & slope & branching set & knot & slope & branching set \\
\hline
$t12533$ & 37 & $K12n407$  & $o9\_40487$ & 37 & $K11n147$ \\
  & 38 & $L12n789$  &  &  38 & $L11n152$ \\
 \hline
$t12681$ & 61 & $K11n89$  & $o9\_40504$ & 58 & $L11n179$ \\
 & 62 & $L11n172$  &  & 59 & $K11n166$ \\
 \hline
 $o9\_38928$ & 49 & $K11n172$  & $o9\_40582$ & 43 & $K13n2958$ \\
  & 50 & $L11n178$  &  & 44 & $L13n4413$ \\
 \hline
$o9\_39162$ & 64 & $L12n1050$  & $o9\_42675$ & 46 & $L12n702$ \\
 & 65 & $K12n278$  & & 47 & $K12n730$\\
 \hline
$o9\_40363$ & 82 & $L12n785$  & &  & \\
 & 83 & $K12n479$  & & & \\
 \hline
\end{tabular}\label{table:QA}
\end{table}

The purpose of this article is to confirm these  $18$ quasi-alternating surgeries by the Montesinos trick \cite{M}.
Needless to say, we cannot apply the Montesinos trick directly for asymmetric knots.
It is necessary to go through strongly invertible surgery diagrams for these surgeries.
This is the most difficult part of the study.
Also, we can say that our argument provides a proof for these asymmetric knots to be L--space knots without relying on computers.

For those asymmetric L--space knots, the braid words are given in \cite{ABG} (and \cite{BK}), but we record them for readers' convenience.
(See Table \ref{table:braid}.)
  The integer $i$ denotes the standard braid generator $\sigma_i$.)
Since each knot is given as the closure of a positive braid,
its quasi-alternating surgeries are positive for this expression.

\begin{table}[ht]
\caption{Braid words for asymmetric census L--space knots.}
\begin{tabular}{ll}
 \hline
 knot & braid word \\
 \hline
$t12533$ & 
[1, 1, 2, 2, 1, 2, 2, 2, 2, 2, 2, 2, 2, 2, 1, 2, 2, 3, 2, 1, 1, 2, 2, 1, 3, 2, 2]\\
 $t12681$ &
[1, 2, 3, 4, 4, 3, 2, 3, 2, 4, 2, 1, 1, 1, 2, 1, 3, 2, 1, 3, 2, 3, 4, 3, 2, 4, 1, \\
&3, 2, 4, 3, 4,  4, 3, 2, 4, 1, 3, 2, 4, 3, 4, 4, 3, 2, 4, 3, 4] \\
 $o9\_38928$ &
[1, 2, 1, 2, 3, 2, 4, 2, 3, 4, 4, 5, 4, 3, 5, 2, 3, 3, 2, 1, 3, 2, 3, 3, 2, 4, 2,\\
&3, 2, 4, 3, 5, 4, 3, 5, 3, 2, 4, 1, 2, 2, 3, 3]\\
$o9\_39162$ &
[1, 1, 2, 1, 3, 2, 4, 2, 5, 1, 3, 2, 2, 3, 2, 4, 2, 5, 2, 4, 3, 2, 4, 2, 5, 4, 3, \\
&5, 2, 4, 2, 5, 2, 4, 3, 2, 2, 3, 3, 2, 2, 1, 2, 3, 4, 3, 2, 4, 5, 4, 3, 4, 3]\\
$o9\_40363$ &
[1, 2, 1, 3, 4, 5, 4, 4, 4, 5, 4, 6, 3, 6, 2, 5, 4, 3, 5, 4, 6, 5, 6, 6, 5, 4, 6, \\
&3, 5, 2, 4, 1, 3, 5, 2, 4, 6, 3, 5, 4, 6, 5, 6, 6, 5, 4, 6, 3, 5, 2, 4, 6, 1, 3, \\
&5, 2, 4, 6, 3, 5, 4, 6, 5, 4, 6, 3, 5, 2, 4, 3, 5, 4]\\
$o9\_40487$ &
[1, 2, 1, 3, 3, 2, 2, 3, 4, 3, 2, 1, 3, 2, 1, 3, 2, 4, 2, 4, 1, 4, 2, 1, 3, 2, 3, \\
&4, 3, 4, 3, 2]\\
%
$o9\_40504$ &
[1, 1, 2, 1, 3, 4, 3, 4, 3, 5, 4, 3, 5, 2, 4, 1, 3, 1, 2, 1, 3, 4, 5, 4, 3, 5, 4, \\
&3, 5, 5, 5, 4, 3, 2, 1, 3, 4, 5, 4, 4, 5, 4, 3, 2, 4, 3, 4]\\
$o9\_40582$ &
[1, 2, 2, 3, 2, 2, 3, 4, 3, 2, 1, 2, 3, 2, 4, 4, 3, 3, 2, 1, 3, 3, 3, 2, 2, 3, 4, \\
&3, 2, 2, 1, 2, 2, 3, 2, 3]\\
 $o9\_42675$  &
[1, 2, 1, 3, 2, 4, 2, 3, 2, 4, 2, 3, 2, 3, 2, 1, 2, 3, 3, 4, 4, 3, 3, 4, 3, 3, 2, \\
&1, 3, 2, 4, 2, 3, 2, 1, 3]\\
 \hline
\end{tabular}\label{table:braid}
\end{table}

Also, surgery descriptions for these asymmetric L--space knots are given in \cite[Table 2]{ABG}.
However, we found that some of them are not suitable for this study.
In this paper, we use the following (Table \ref{table:surgery}).
For the notation, see Section \ref{sec:t12533}.
The diagrams of these links are illustrated in Fig.~\ref{fig:start}.

\renewcommand{\arraystretch}{1.2}
\begin{table}[ht]
\caption{Surgery descriptions for asymmetric census L--space knots.
The knot corresponds to the component with $*$.}
\begin{tabular}{lll}
 \hline
knot & surgery description & \\
\hline
$t12533$ & $L14n58444(\frac{5}{2},\frac{1}{2},*)$ & $L12n1638(-\frac{2}{3},*,-1)$ \\
$t12681$ & $L12n1968(\frac{1}{3},\frac{7}{2},*)$ &  $L14n63000(\frac{2}{3},*,\frac{1}{2},1)$ \\
 $o9\_38928$ &  $L11n456(-\frac{3}{4},-2, *,-2)$ & $L13n8037(-\frac{1}{3}, -2, *)$      \\
 $o9\_39162$ &  $L12n1968(\frac{1}{4},5,*)$ & $L13n9366(*,\frac{3}{4},5)$     \\
 $o9\_40363$ &  $L14n63000(\frac{3}{4},*, \frac{1}{2},1)$ & $L12n1968(*,\frac{7}{2},\frac{1}{4})$    \\
 $o9\_40487$ &  $L13n8037(-\frac{1}{2},-1, *)$ &   $L12n1625(-\frac{1}{3},-4,*)$  \\
 $o9\_40504$ &    $L14n62791(-\frac{1}{3},-2,\frac{1}{2},*)$ &  $L11n456(-\frac{3}{5},-2,*,-2)$     \\
 $o9\_40582$ &      $L12n1638(-\frac{2}{3},*,-2)$ &    $L14n58444(-1,-2,*)$   \\
  $o9\_42675$  &       $L12n1625(-\frac{1}{4},-3,*)$ &  $L14n63014(-4,-\frac{3}{2},*,-1)$       \\
 \hline
\end{tabular}\label{table:surgery}
\end{table}

\renewcommand{\arraystretch}{1}

Although these surgery descriptions can be confirmed in SnapPy, 
we describe a proof by hand.

\begin{figure}[th]
\centerline{\includegraphics[bb=0 0 559 710, width=12cm]{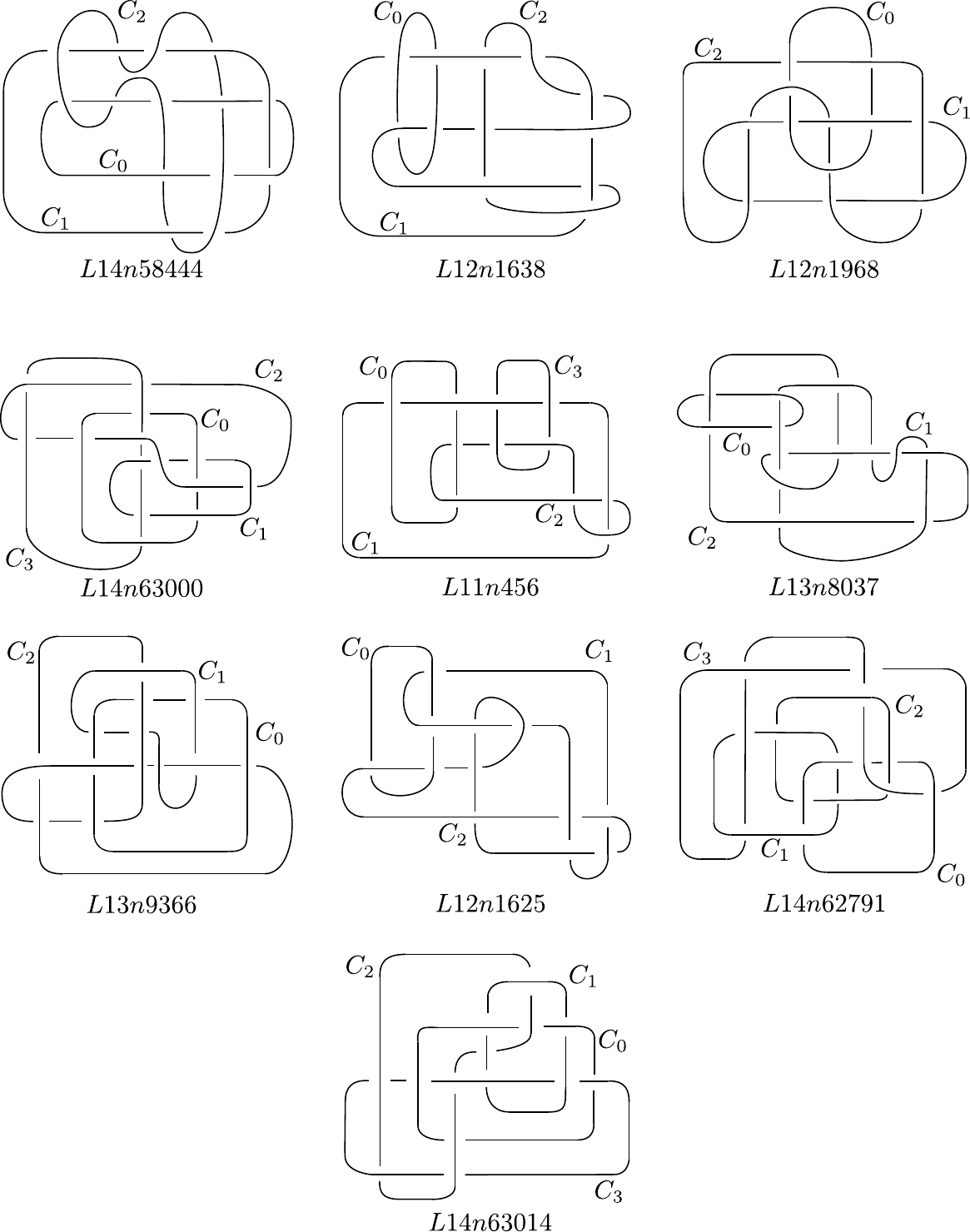}}
\vspace*{8pt}
\caption{The list of links.
The order of components corresponds to SnapPy.
\label{fig:start}}
\end{figure}


\section{$t12533$}\label{sec:t12533}

\subsection{$37$--surgery}

We start from the link $L14n58444$ with components $C_0$, $C_1$ and $C_2$ as illustrated in Fig.~\ref{fig:start}.
The order of the components is $(C_0,C_1,C_2)$ as matched with the diagram in the census.
The notation $L14n58444(\frac{5}{2},\frac{1}{2},-1)$ indicates the surgery diagram whose
coefficients are assigned to the ordered components.


\begin{lemma}\label{lem:t12533-37knot}
For the surgery diagram $L14n58444(\frac{5}{2},\frac{1}{2},-1)$,
performing $\frac{5}{2}$--surgery on $C_0$ and $\frac{1}{2}$--surgery on $C_1$
changes $C_2$  into the mirror image of  $t12533$.
Hence, the surgery diagram $L14n58444(\frac{5}{2},\frac{1}{2},-1)$
 represents $(-37)$--surgery on the mirror of $t12533$.
\end{lemma}

\begin{proof}
After $-2$ twists along $C_1$ to erase it,  the surgery coefficient on $C_2$ (resp. $C_0$) is changed to $-19$ (resp. $\frac{1}{2}$).  See Fig.~\ref{fig:t12533-37}.
(We keep the same name for remaining components after twisting throughout the paper.
Also, a box with integer $i$ indicates right handed or left handed $|i|$ full twists on the strands, according to the sign of $i$.)
Then $-2$ twists along $C_0$ changes the coefficient on $C_2$ to $-37$. 
Furthermore, we can deform $C_2$ to the closure of a $4$--braid
\[
[(-3,-2,-1)^8,-2,-2,-1,3,2,-1,-1].
\]
as shown in Fig.~\ref{fig:t12533-37-2}.
Here, a negative integer  $i$ indicates the inverse $\sigma_i^{-1}$ and a power means a repetition.
Clearly, this is conjugate to
\begin{equation}\label{eq:t12533-37}
[-1,-1,-2,-2,-1,3,2,(-3,-2,-1)^8].
\end{equation}
Note that  $(-3,-2,-1)^8$  corresponds to $(-2)$--full twists, so it is central.

\begin{figure}[th]
\centerline{\includegraphics[bb=0 0 430 173, width=10cm]{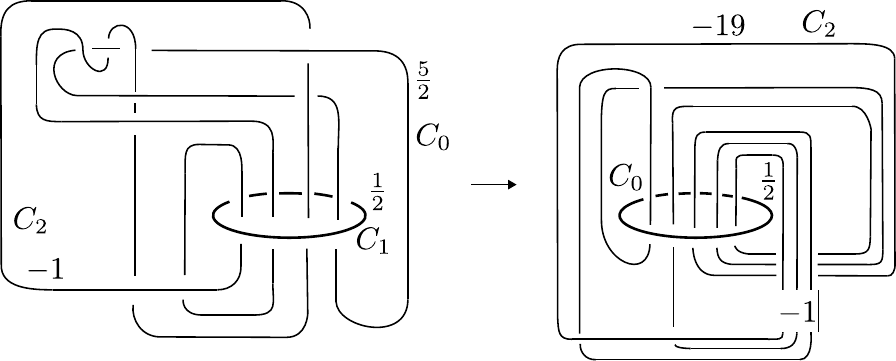}}
\vspace*{8pt}
\caption{For the link diagram $L14n58444(\frac{5}{2},\frac{1}{2},-1)$, do $-2$ twists on $C_1$ to erase it.
The box with integer $-1$ indicates $-1$  (left handed) full twist on the strands.
\label{fig:t12533-37}}
\end{figure}

\begin{figure}[th]
\centerline{\includegraphics[bb=0 0 514 331, width=12cm]{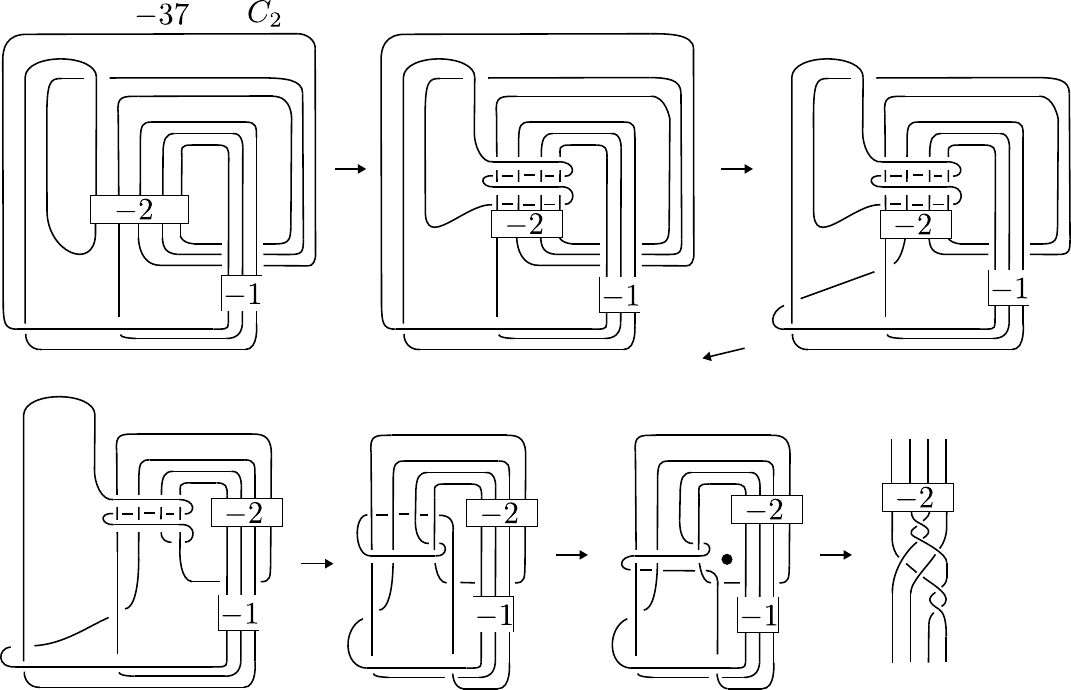}}
\vspace*{8pt}
\caption{$C_2$ is deformed into the closure of a $4$--braid.
A box with a negative integer $i$ contains $i$ full twists.
A black dot indicates the axis of the closed braid.
\label{fig:t12533-37-2}}
\end{figure}

We need to identify it as the mirror of $t12533$.
In Table \ref{table:braid},
the braid word of  $t12533$ is
\[
\beta=[1, 1, 2, 2, 1, 2, 2, 2, 2, 2, 2, 2, 2, 2, 1, 2, 2, 3, 2, 1, 1, 2, 2, 1, 3, 2, 2].
\]

Set $\alpha=[-2,-2,-3,-1,-2,3,3,-1]$.
Then we see 
\[
\alpha^{-1}\beta \alpha=[3,3,2,2,3,-1,-2,(1,2,3)^8].
\]

%
By renaming the generators $\sigma_i$ with $\sigma_{4-i}$ $(\, i=1,2,3)$,
this gives 
\[
[1,1,2,2,1,-3,-2,(3,2,1)^8],
\]
which is the mirror of (\ref{eq:t12533-37}).
\end{proof}


\begin{theorem}\label{lem:t12533-37}
The diagram $L14n58444(\frac{5}{2},\frac{1}{2},-1)$
represents the double branched cover of $K12n407$.
\end{theorem}

\begin{proof}
After performing $+1$ twist on $C_2$ first, we obtain the diagram  in a strongly invertible position
as in Fig.~\ref{fig:v3437}.
The coefficients of $(C_0,C_1)$ are changed to $(\frac{23}{2},\frac{19}{2})$.

\begin{figure}[th]
\centerline{\includegraphics[bb=0 0 529 185, width=12cm]{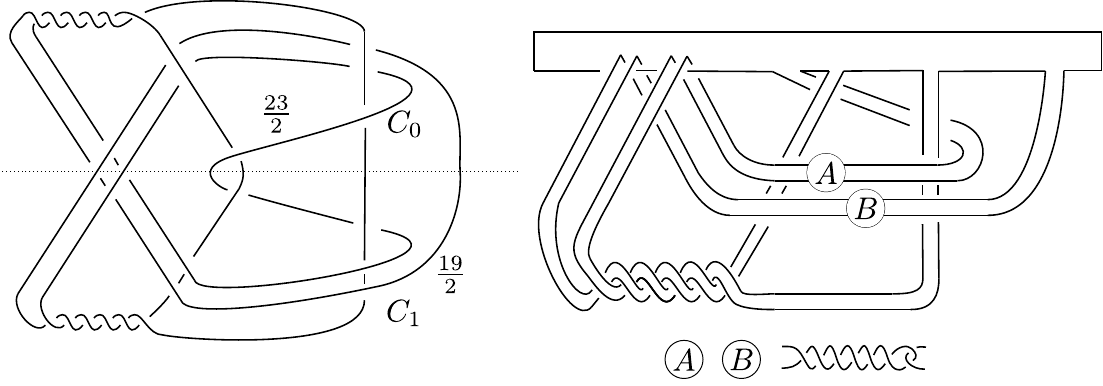}}
\vspace*{8pt}
\caption{Left: A strongly invertible position of the link after $+1$ twist on $C_2$ of $L14n58444$.
Right: The knot after the tangle replacement. The tangles $A$ and $B$ are the rational tangle $[6,-2]$.
\label{fig:v3437}}
\end{figure}

Hence we can carry out the Montesinos trick.
Note that $C_0$ and $C_1$ have writhe $5$ and $3$, respectively in the diagram.
After taking the quotient around the axis of the involution,
the tangle replacement corresponding to $C_0$ is $\frac{23}{5}-5=\frac{13}{2}=[6,-2]$.
Similarly, the tangle for $C_1$ is $\frac{19}{2}-3=\frac{13}{2}=[6,-2]$.
Thus both tangles $A$ and $B$ in Fig.~\ref{fig:v3437} are the rational tangle $[6,-2]$.
Then there is a series of deformations to confirm that the knot is the mirror of  $K12n407$.  See Fig.~\ref{fig:K12n407}.
\end{proof}


\begin{figure}[th]
\centerline{\includegraphics[bb=0 0 578 403, width=12cm]{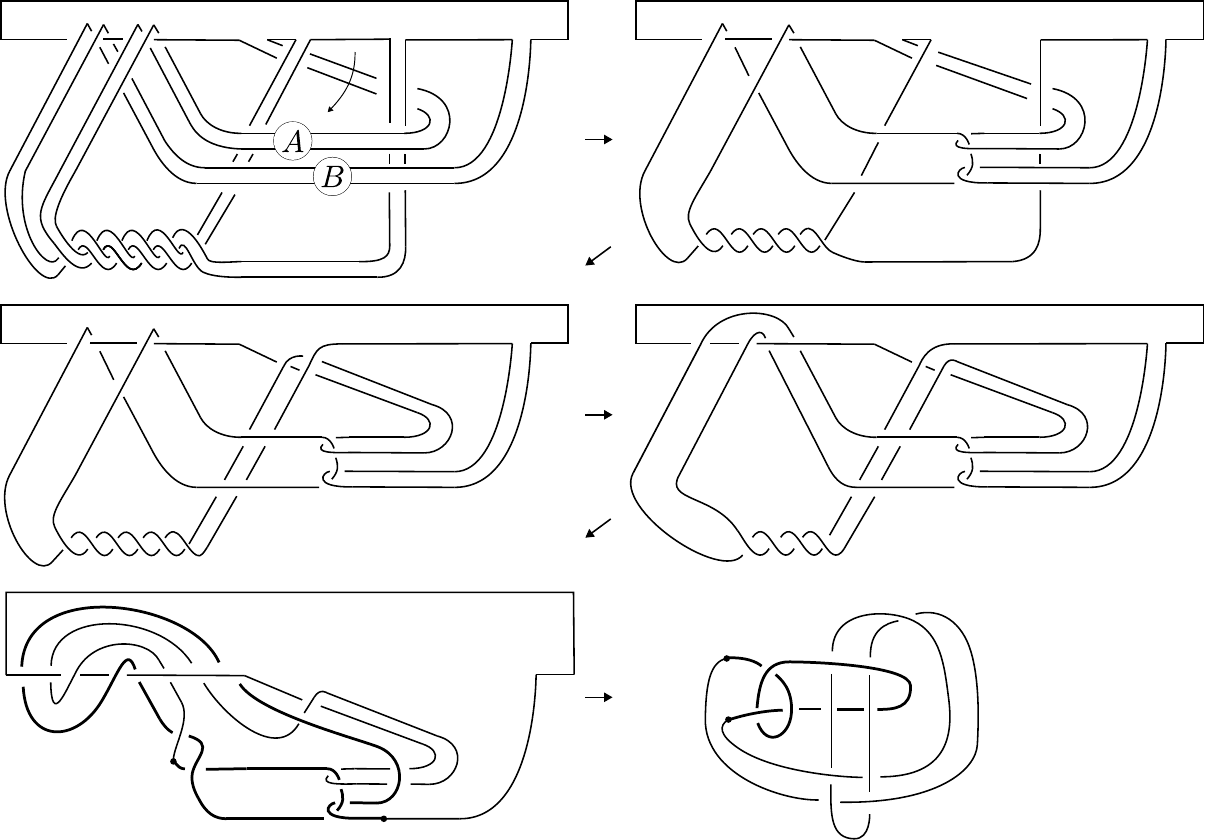}}
\vspace*{8pt}
\caption{A series of deformations from the knot in Fig.~\ref{fig:v3437} to the mirror of $K12n407$ (Bottom Right).
\label{fig:K12n407}}
\end{figure}

\begin{remark}
Our starting link $L14n58444$, is not strongly invertible, because
a partial surgery $L14n58444(\frac{5}{2},\frac{1}{2},*)$ yields  the asymmetric knot $t12533$.
(Here, $*$ means no surgery on the component.)
Hence we cannot apply the Montesinos trick to the diagram of $L14n58444$ directly.

For the others in the rest of the paper, the situation is similar.
\end{remark}

\subsection{$38$--surgery}\label{subsec:t12533-38}

For this case, we use the link $L12n1638$  shown in Fig.~\ref{fig:start}.


\begin{lemma}
$L12n1638(-\frac{2}{3}, 1, -1)$ represents $(-38)$--surgery on  the mirror of $t12533$.
\end{lemma}

\begin{proof}
 First, do $+1$ twist on $C_2$, which changes the coefficient on $C_0$ to $\frac{1}{3}$.  See Fig.~\ref{fig:t12533-38}.
Then performing $-3$ twists on $C_0$ immediately yields a knot which lies in a closed braid form and has coefficient $-38$.

With keeping it as a closed braid, we can deform it
as shown in Fig.~\ref{fig:t12533-38braid}.
The last braid is conjugate to (\ref{eq:t12533-37}).
\end{proof}

\begin{figure}[th]
\centerline{\includegraphics[bb=0 0 401 168, width=9cm]{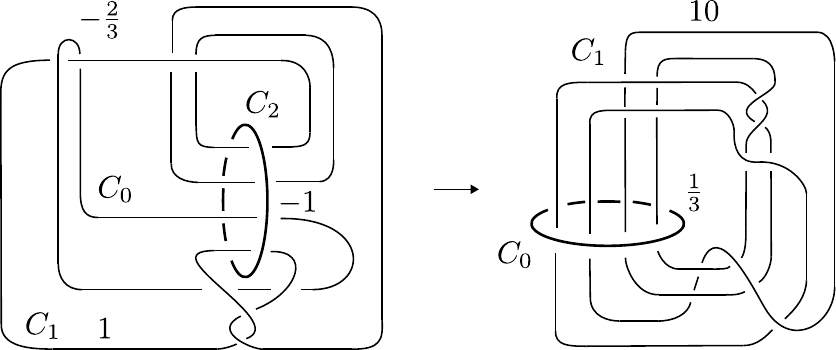}}
\vspace*{8pt}
\caption{For $L12n1638(-\frac{2}{3},1,-1)$, do $+1$ twist on $C_2$.
\label{fig:t12533-38}}
\end{figure}

\begin{figure}[th]
\centerline{\includegraphics[bb=0 0  487 315, width=12cm]{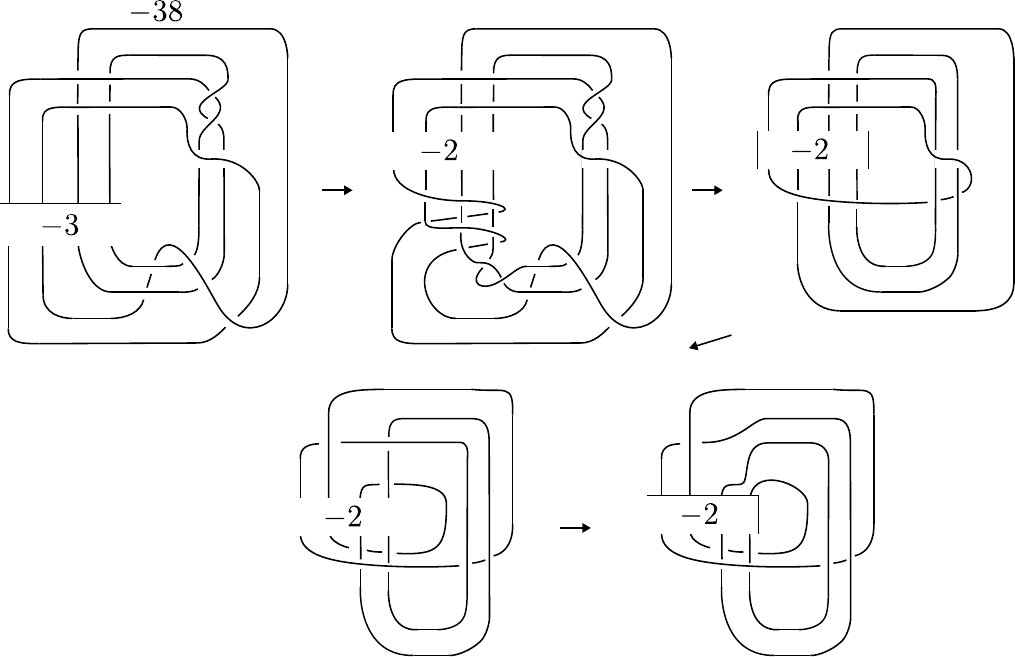}}
\vspace*{8pt}
\caption{$-3$ twists on $C_0$ yields a knot in a closed braid form.
\label{fig:t12533-38braid}}
\end{figure}

\begin{theorem}\label{lem:t12533-38}
The diagram $L12n1638(-\frac{2}{3},1,-1)$
represents the double branched cover of $L12n789$.
\end{theorem}

\begin{proof}
After $-1$ twist on $C_1$, we obtain the diagram in a strongly invertible position as in Fig.~\ref{fig:L14n24287}.
(In fact, this is the mirror of $L14n24287(\frac{5}{3},10)$.)

Take the quotient around the axis.  Then the tangle replacement for the image of $C_0$ is by the rational tangle
$-\frac{5}{3}=[-1,1,-2]$, but that of $C_2$ is $-8$, because $C_2$ has writhe $-2$ in the original diagram
(Fig.~\ref{fig:L14n24287}).

We can identify that the result is the link $L12n789$ as shown  in Fig.~\ref{fig:L12n789}.
Two links consist of only unknotted components.
By examining how one component wraps around the other,
we have the equivalence between these links.
\end{proof}

\begin{figure}[th]
\centerline{\includegraphics[bb=0 0 498 112, width=12cm]{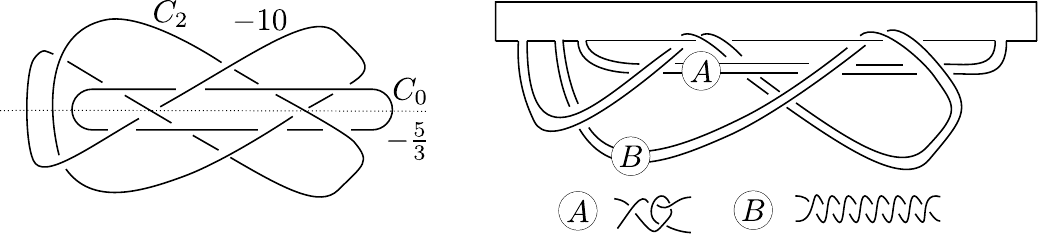}}
\vspace*{8pt}
\caption{Left: A strongly invertible position of the link after $-1$ twist on $C_1$.
Right: The link after tangle replacement.
\label{fig:L14n24287}}
\end{figure}


\begin{figure}[th]
\centerline{\includegraphics[bb=0 0 542 282, width=12.5cm]{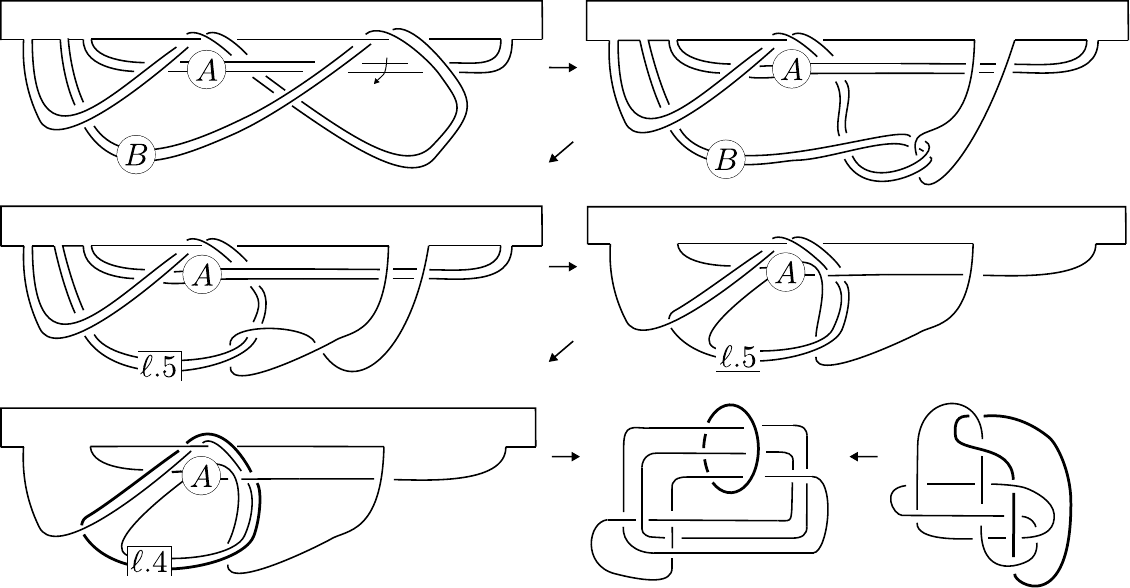}}
\vspace*{8pt}
\caption{The equivalence between the link in Fig.~\ref{fig:L14n24287} and $L12n789$ (Bottom Right).
A box labeled with $\ell.i$ contains left handed  horizontal $i$ half twists.
The  corresponding unknotted components are emphasized by ticker lines.
\label{fig:L12n789}}
\end{figure}

\section{$t12681$}\label{sec:t12681}

\subsection{$61$--surgery}

As in \cite{ABG}, we use the link $L12n1968$.  
See Fig.~\ref{fig:start}.

\begin{lemma}
$L12n1968(\frac{1}{3}, \frac{7}{2}, 1)$ represents $(-61)$--surgery on  the mirror of $t12681$.
\end{lemma}

\begin{proof}
Do $-3$ twists on $C_0$.  This yields the diagram in Fig.~\ref{fig:t12681-61} (Right).
Then $-2$ twists on $C_1$ changes $C_2$ into 
a knot which is the closure of a $5$--braid
\[
[(-4,-3,-2,-1)^{10},-3,-3,-3,-4,-3,-3,-3,-3,-2,1],
\]
Its coefficient is $-61$.
Deform the knot as shown in Fig.~\ref{fig:t12681-61braid}.
Let 
\begin{equation}\label{eq:t12681-61}
\beta=[(1,2,3,4)^{10}, -1,3,2,2,2,2,3,2,4,3].
\end{equation}
Thus our knot is the mirror of the closure of $\beta$.

\begin{figure}[th]
\centerline{\includegraphics[bb=0 0 378 166, width=9cm]{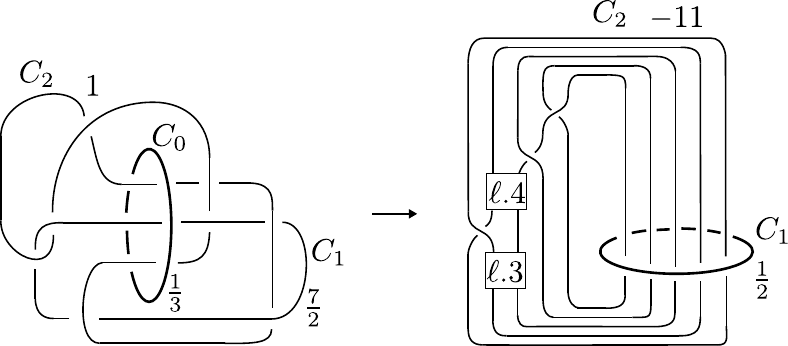}}
\vspace*{8pt}
\caption{For $L12n1968(\frac{1}{3},\frac{7}{2},1)$, do $-3$ twists on $C_0$.
A box labeled with $\ell. i$ indicates left handed vertical $i$ half twists.
\label{fig:t12681-61}}
\end{figure}

\begin{figure}[th]
\centerline{\includegraphics[bb=0 0 403 145, width=12cm]{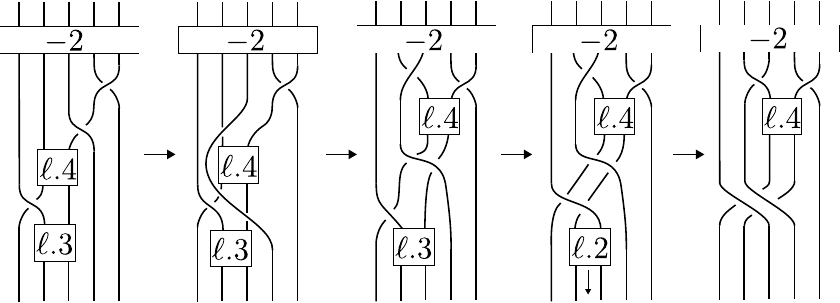}}
\vspace*{8pt}
\caption{Deform $C_2$ after $-2$ twists on $C_1$.
\label{fig:t12681-61braid}}
\end{figure}

We will verify that the closure of $\beta$ is $t12681$.
In Table \ref{table:braid}, $t12681$ is the closure of 
\[
\begin{split}
\gamma=[&1, 2, 3, 4, 4, 3, 2, 3, 2, 4, 2, 1, 1, 1, 2, 1, 3, 2, 1, 3, 2, 3, 4, 3, 2, 4, 1, 3, 2, 4, 3, 4, \\
& 4, 3, 2, 4, 1, 3, 2, 4, 3, 4, 4, 3, 2, 4, 3, 4].
\end{split}
\]
Set $\alpha=[-2,-4,-4,-4,-3,-2,-4,-3]$.
Then we see 
\[
\alpha^{-1}\gamma\alpha=[(1,2,3,4)^{10},2,-4,2,1,3,2,3,3,3,3].
\]

Note that $(1,2,3,4)^{10}$ corresponds to $2$ full twists, which is central.
By renaming the generators $\sigma_i$ with $\sigma_{5-i}$, 
it changes to $[(4,3,2,1)^{10}, 3,-1,3,4,2,3,2,2,2,2]$.
If we reverse the order of words, we can conclude that $\beta$ and $\gamma$ give the same closure $t12681$.
\end{proof}

\begin{theorem}\label{lem:t12681dbc}
The diagram $L12n1968(\frac{1}{3},\frac{7}{2},1)$ represents the double branched cover of
$K11n89$.
\end{theorem}

\begin{proof}
For the diagram $L12n1968(\frac{1}{3},\frac{7}{2},1)$, perform $-1$ twist on $C_2$.
The result is a strongly invertible diagram as in Fig.~\ref{fig:L9n14}, which represents the pretzel link $P(2,4,-3)$.
(In SnapPy, this is $L9n14(\frac{5}{2},-\frac{11}{3})$.)

After taking the quotient, the tangle replacements are $-\frac{11}{3}+3=-\frac{2}{3}=[0,1,-2]$ for the image of $C_0$ and $\frac{5}{2}=[2,-2]$ for that of $C_1$.  See Fig.~\ref{fig:L9n14}.
For this case, it is not hard to see that the resulting knot is
 the Montesinos knot $(\frac{3}{5},\frac{2}{3},-\frac{1}{4})$, which is the mirror of $K11n89$.
\end{proof}

\begin{figure}[th]
\centerline{\includegraphics[bb=0 0 450 141, width=12cm]{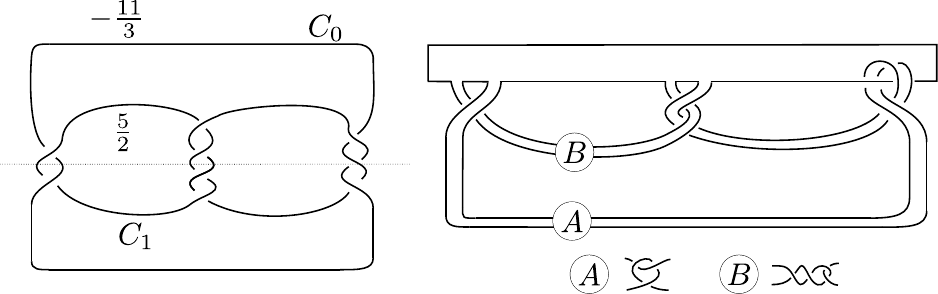}}
\vspace*{8pt}
\caption{Left: A strongly invertible diagram of the link after $-1$ twist on $C_2$.
Right:  The knot after tangle replacement.
This is the Montesinos knot $(\frac{3}{5},\frac{2}{3},-\frac{1}{4})$, which is the mirror of $K11n89$.
\label{fig:L9n14}}
\end{figure}



\subsection{$62$--surgery}


\begin{lemma}\label{lem:t12681-62}
$L14n63000(\frac{2}{3}, 1,\frac{1}{2}, 1)$ represents $62$--surgery on $t12681$.
\end{lemma}

\begin{proof}
Do $-1$ twist on $C_3$, and then $3$ twists on $C_0$.  See Fig.~\ref{fig:t12681-62}.
Finally, performing $2$ twists on $C_2$ changes $C_1$ into
a knot in a closed braid form with coefficient $62$.
Its braid is $[(1,2,3,4)^{10}, 3,4,2,3,2,2,2,3,-1,4]$, which is conjugate to $[(1,2,3,4)^{10},3,4,2,3,2,2,2,2,3,-1]$
by $\sigma_4^{-1}$.
Then this braid shares the same closure as $\beta$ defined in (\ref{eq:t12681-61}).

\end{proof}

\begin{figure}[th]
\centerline{\includegraphics[bb=0 0 525 119, width=12cm]{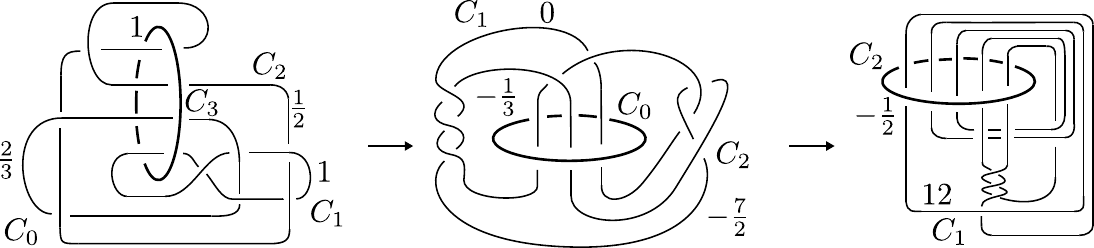}}
\vspace*{8pt}
\caption{Do $-1$ twist on $C_3$ and $3$ twists on $C_0$.
\label{fig:t12681-62}}
\end{figure}

\begin{figure}[th]
\centerline{\includegraphics[bb=0 0 206 109, width=5cm]{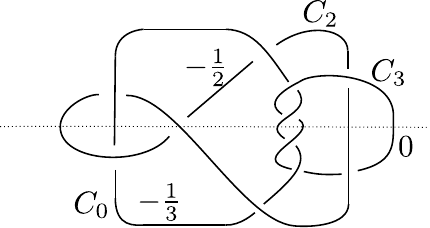}}
\vspace*{8pt}
\caption{A strongly invertible position of the link after $-1$ twist on $C_1$.
\label{fig:L14n63000}}
\end{figure}

\begin{theorem}\label{thm:t12681-62}
The diagram $L14n63000(\frac{2}{3},1,\frac{1}{2},1)$ represents the double branched cover of
$L11n172$.
\end{theorem}

\begin{proof}
In the diagram $L14n63000(\frac{2}{3},1,\frac{1}{2},1)$, perform $-1$ twist on $C_1$ first.
This yields a strongly invertible diagram as in Fig.~\ref{fig:L14n63000}, which is the mirror of $L10n92(\frac{1}{2},0,\frac{1}{3})$ in the SnapPy notation.


Take the quotient of this diagram under the involution.
Figure \ref{fig:L10n92q} shows the link after the tangle replacement.  By modifying it indicated there,
we can easily see the unknotted component.
The other component is the knot $5_2$.

\begin{figure}[th]
\centerline{\includegraphics[bb=0 0 391 123, width=12cm]{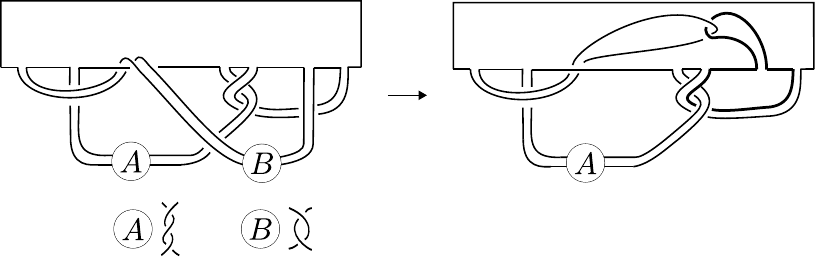}}
\vspace*{8pt}
\caption{The tangle replacement and its deformation.
The unknotted component is drawn in  a thicker line.
\label{fig:L10n92q}}
\end{figure}

We identify  this link as the mirror of $L11n172$.
For both links, look at how the knotted component wraps around the unknotted component
as shown in Fig.~\ref{fig:L11n172}.
\end{proof}

\begin{figure}[th]
\centerline{\includegraphics[bb=0 0 468 115, width=12cm]{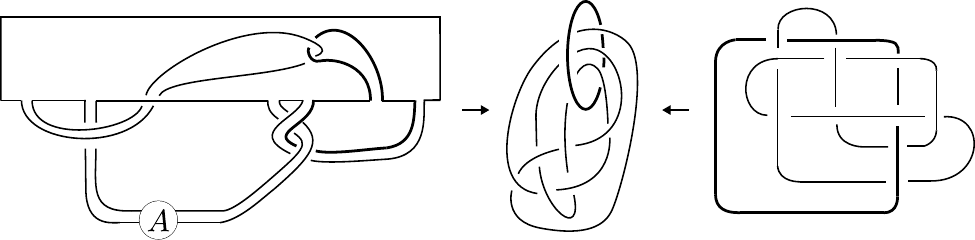}}
\vspace*{8pt}
\caption{The equivalence between the link of Fig.~\ref{fig:L10n92q} and the mirror of $L11n172$.
Right is the mirror of $L11n172$.
\label{fig:L11n172}}
\end{figure}

\section{$o9\_38928$}

\subsection{$49$--surgery}\label{subsec:o9_38928-49}

\begin{lemma}
$L11n456(-\frac{3}{4}, -2,-1,-2)$ represents $49$--surgery on $o9\_38928$.
\end{lemma}

\begin{proof}
Perform $+1$ twist on $C_3$, which changes the coefficient of $C_1$ to $-1$
(Fig.~\ref{fig:o9_38928-49}).
Do $+1$ twist on $C_1$ to erase it.
Then $C_0$ has coefficient $\frac{1}{4}$, so do $-4$ twists there (Fig.~\ref{fig:o9_38928-49-2}).
At this point, $C_2$ and $C_3$ have coefficients $-32$ and $-1$, respectively.
Finally, performing $+1$ twist on $C_3$ yields a knot with coefficients $49$.

\begin{figure}[th]
\centerline{\includegraphics[bb=0 0 459 178, width=10cm]{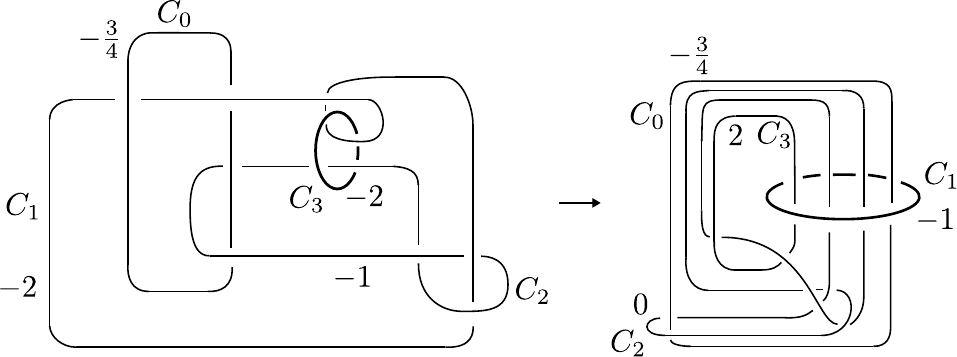}}
\vspace*{8pt}
\caption{Do $+1$ twist on $C_3$.
\label{fig:o9_38928-49}}
\end{figure}

\begin{figure}[th]
\centerline{\includegraphics[bb=0 0 392 203, width=10cm]{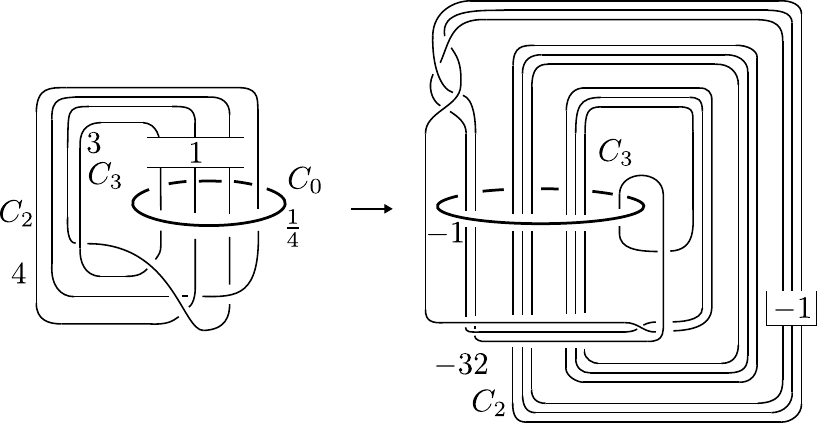}}
\vspace*{8pt}
\caption{Do $-4$ twists on $C_0$.
\label{fig:o9_38928-49-2}}
\end{figure}

We reduce this knot into the closure of  a $6$--braid  as shown in Fig.~\ref{fig:o9_38928-49-3}.
Furthermore, we deform the braid into $\beta$ as in Fig.~\ref{fig:o9_38928-49braid}, where
\[
\beta=[(1,2,3,4,5)^{6},2,2,3,2,4,3,5,4,2,1,3,2,4].
\]

By taking a conjugation with  $\alpha=[-4,-5,3]$,
\[
\alpha^{-1}\beta \alpha=
[(1,2,3,4,5)^6,5,4,3,2,1,3,4,5,4,2,1,3,2].
\]

\begin{figure}[th]
\centerline{\includegraphics[bb=0 0 641 440, width=12cm]{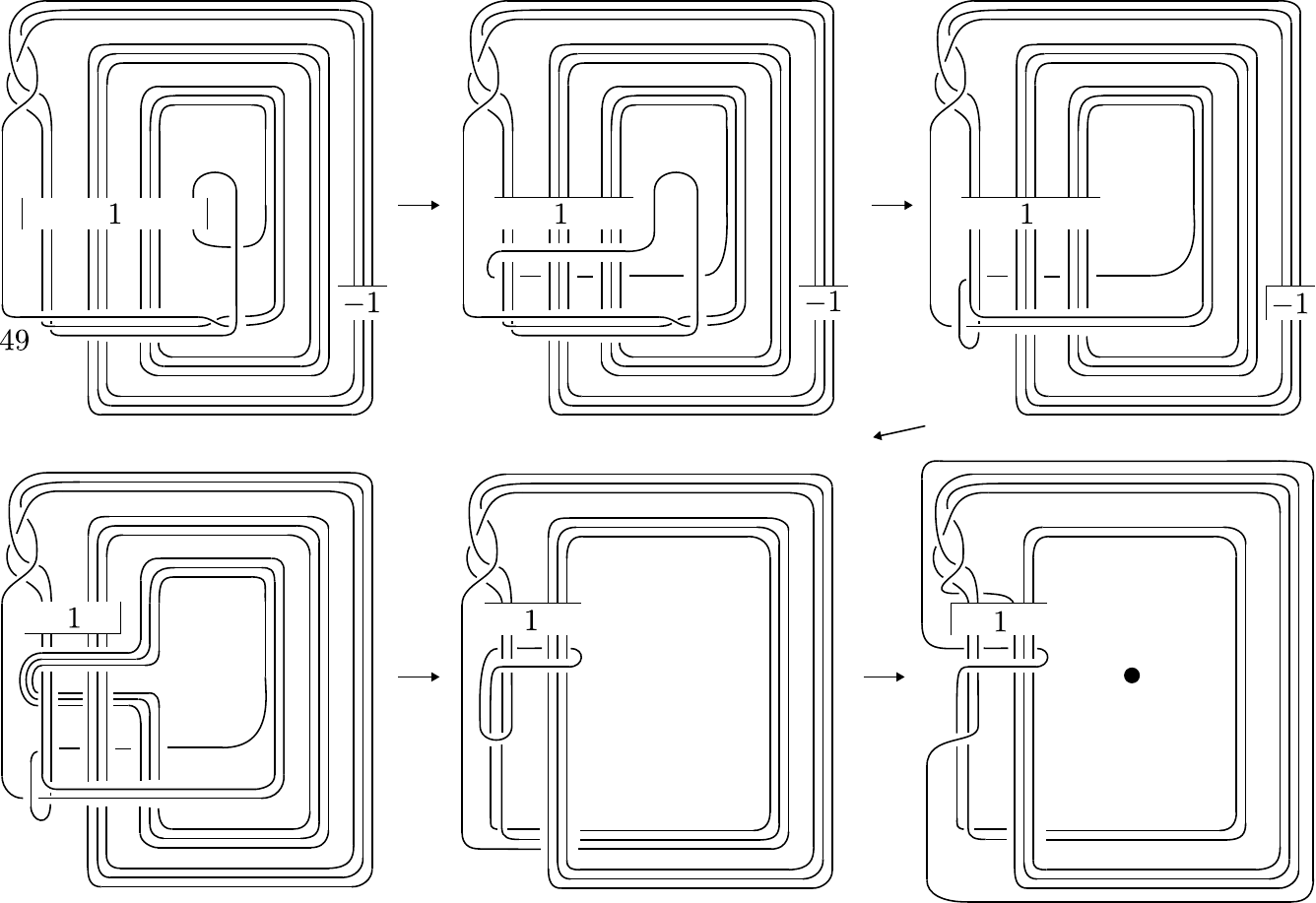}}
\vspace*{8pt}
\caption{Deform the knot into the closure of a $6$--braid.
\label{fig:o9_38928-49-3}}
\end{figure}

\begin{figure}[th]
\centerline{\includegraphics[bb=0 0 623 216, width=12cm]{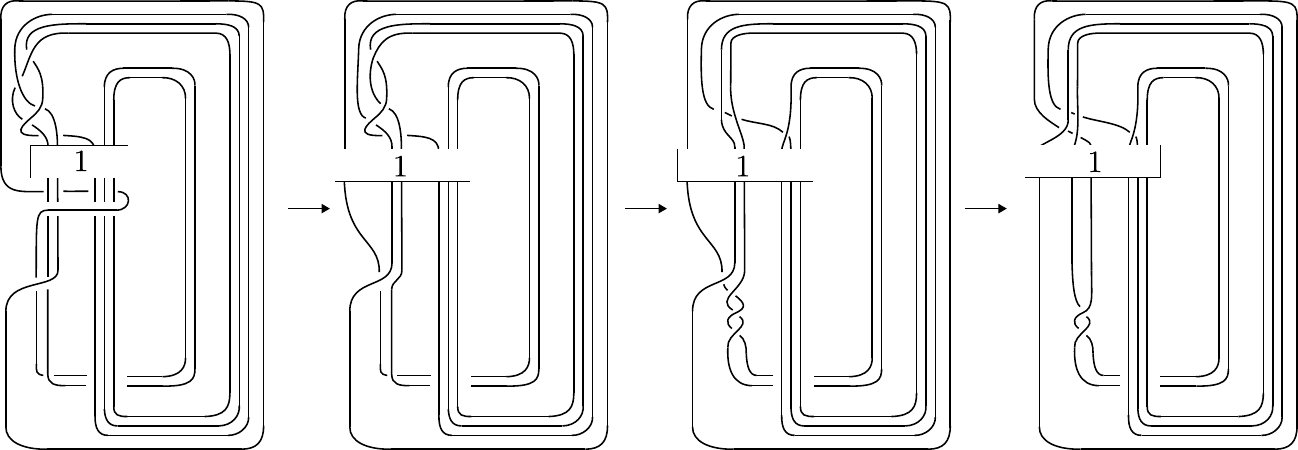}}
\vspace*{8pt}
\caption{Further deformation of the braid.
\label{fig:o9_38928-49braid}}
\end{figure}

We identify this knot (the closure of $\beta$) as $o9\_38928$.
From Table \ref{table:braid}, we have $o9\_38928$ as the closure of a braid
\[
\begin{split}
\gamma&=[1, 2, 1, 2, 3, 2, 4, 2, 3, 4, 4, 5, 4, 3, 5, 2, 3, 3, 2, 1, 3, 2, 3, 3, 2, 4, 2, 3, 2, 4, 3, 5, \\
&\qquad  4, 3, 5, 3, 2, 4, 1, 2, 2, 3, 3].
\end{split}
\]
Set $\alpha'=[-3,-3,-2,-1,5,3]$.
Then
\begin{equation}\label{eq:o9_38928}
\alpha'^{-1}\gamma \alpha'=
[(1,2,3,4,5)^6, 1,2,3,4,5,2,3,1,2,4,5,4,3].
\end{equation}
Reversing the order of words,  we can conclude that the closure of $\beta$ is $o9\_38928$.
\end{proof}

\begin{theorem}\label{lem:o9_38928dbc}
The diagram $L11n456(-\frac{3}{4},-2,-1,-2)$ represents the double branched cover of $K11n172$.
\end{theorem}

\begin{proof}
Performing $+1$ twist on $C_2$ yields a strongly invertible diagram as shown in Fig.~\ref{fig:L12n1896}, which is the mirror of $L12n1896(-\frac{1}{4},1,1)$ in  SnapPy.
Take the quotient, then the tangle replacement as indicated in Fig.~\ref{fig:L12n1896} yields the mirror of the knot $K11n172$.
This can be seen by rewriting two knots as the closure of a braid
as shown in Fig.~\ref{fig:K11n172}.
\end{proof}

\begin{figure}[th]
\centerline{\includegraphics[bb=0 0 483 166, width=12cm]{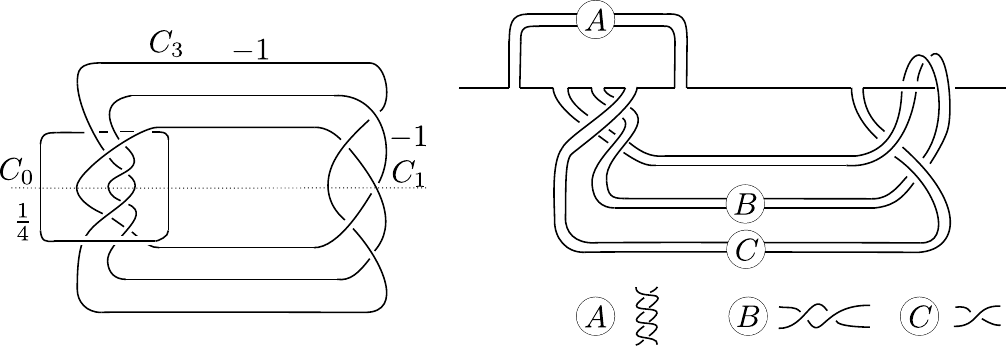}}
\vspace*{8pt}
\caption{Left: A strongly invertible diagram after $+1$ twist on $C_2$.
Right:  The tangle replacement.
\label{fig:L12n1896}}
\end{figure}


\begin{figure}[th]
\centerline{\includegraphics[bb=0 0 542 377, width=12.5cm]{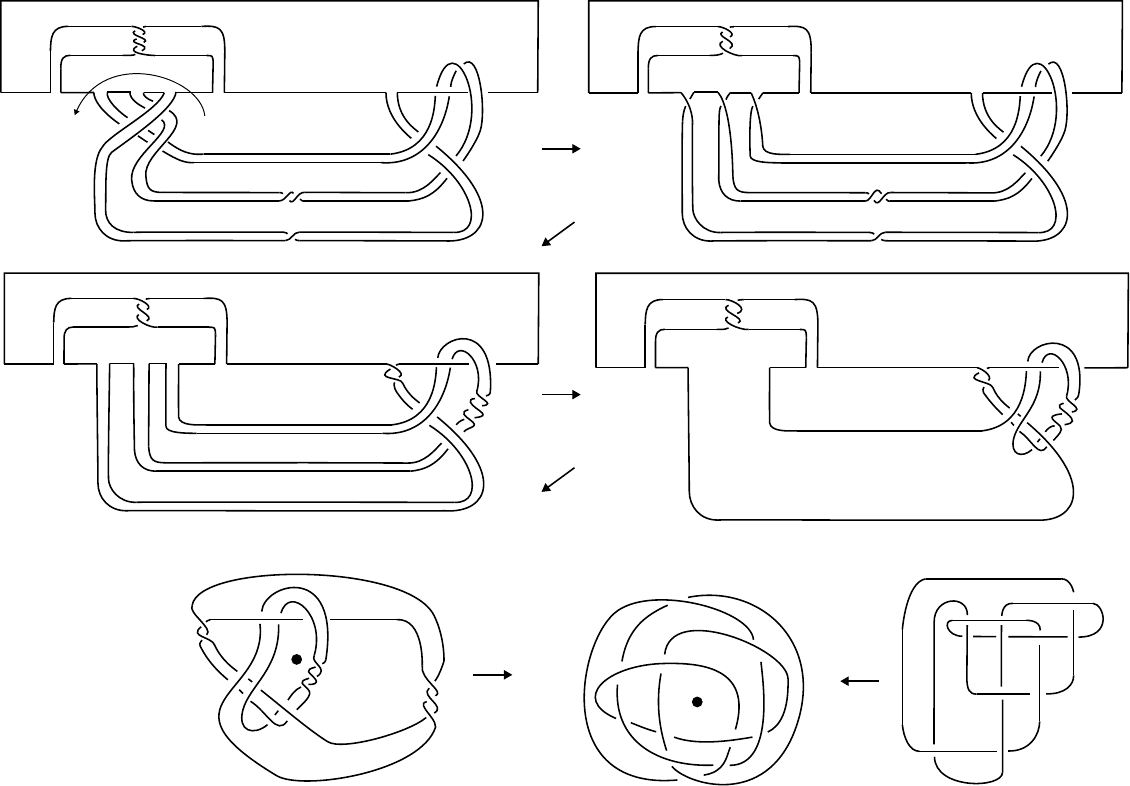}}
\vspace*{8pt}
\caption{The equivalence between the knot in Fig.\ref{fig:L12n1896} and the mirror of $K11n172$ (Bottom Right).
Both are expressed as the closure of the same braid.
(Black dots indicate the axes of closed braids.)
\label{fig:K11n172}}
\end{figure}


\subsection{$50$--surgery}

\begin{lemma}
$L13n80367(-\frac{1}{3}, -2,-4)$ represents $(-50)$--surgery on the mirror of $o9\_38928$.
\end{lemma}

\begin{proof}
Perform $3$ twists on $C_0$ to erase it.  See Fig.~\ref{fig:L13n8037}.
Then $C_1$ and $C_2$ have coefficients $1$ and $-1$, respectively.
Performing $-1$ twist on $C_1$ changes $C_2$ into a knot with coefficient $-50$.
It has a closed braid form of $7$ strands, but we can deform it into that of a $6$--braid
as shown in Fig.~\ref{fig:o9_38928-50}.
Let
\[
\beta=[(1,2,3,4,5)^6,1,2,3,4,5,5,4,3,3,4,2,3,5].
\]
Then our knot is the mirror of the closure of $\beta$.

We can verify that $\beta$ and $o9\_38928$ share the same closure.
By taking a conjugation with $\alpha=[-5,-3,-2,-4,-3]$,
\[
\alpha^{-1}\beta \alpha
=[(1,2,3,4,5)^6,	1,2,3,4,5,2,3,1,2,4,5,4,3],
\]
which is equal to (\ref{eq:o9_38928}).
\end{proof}

\begin{figure}[th]
\centerline{\includegraphics[bb=0 0 381 160, width=10cm]{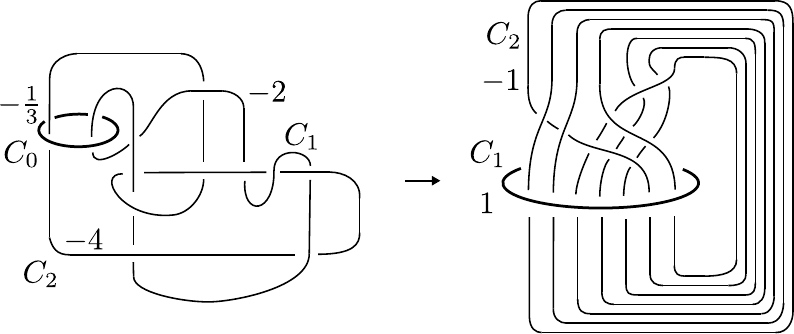}}
\vspace*{8pt}
\caption{
 Do $3$ twists on $C_0$.
 \label{fig:L13n8037}}
\end{figure}

\begin{figure}[th]
\centerline{\includegraphics[bb=0 0 636 174, width=12cm]{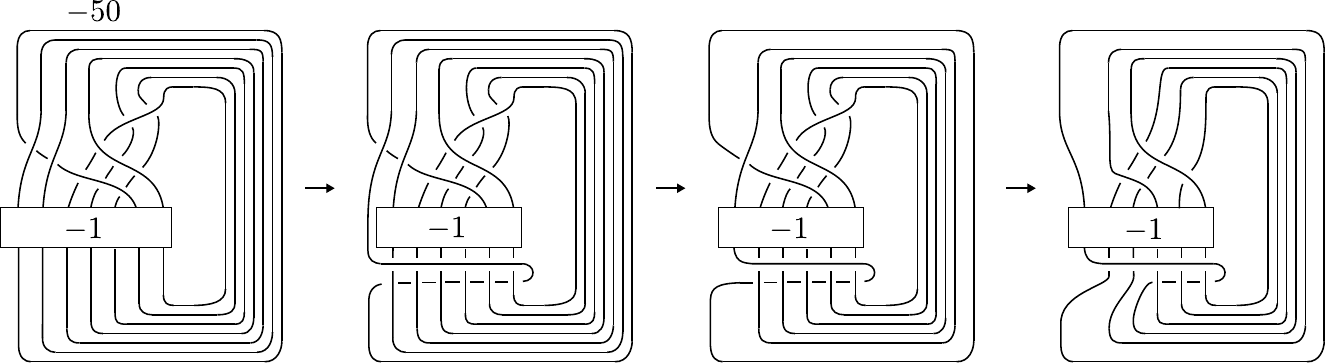}}
\vspace*{8pt}
\caption{After $-1$ twist on $C_1$, we deform the knot into a closed braid form of $6$ strands.
 \label{fig:o9_38928-50}}
\end{figure}


\begin{theorem}
The diagram $L13n8037(-\frac{1}{3},-2,-4)$ represents the double branched cover of $L11n178$.
\end{theorem}

\begin{proof}
Perform $3$ twists on $C_0$.  The resulting link consists of  two unknotted components $C_1$ and $C_2$, where $C_1$ wraps around $C_2$ in the same direction.  See Fig.~\ref{fig:L13n8037}.
This means that $C_1$ is braided in the complement of $C_2$.
Note that $C_1$ and $C_2$ have coefficients $1$ and $-1$, respectively.
Performing $+1$ twist on $C_2$ changes $C_1$ with coefficient $50$.
Furthermore,  we can see that this $C_1$ is the closure of a positive  $5$--braid
\[
\begin{split}
&2,2,2,4,(3,2,4,3),(1,2,3,4),(4,3,2,1),3,3,2,4,\\
& 1,3,4,2,3,3,(1,2,3,4),(4,3,2,1),(3,2,4,3).
\end{split}
\]
(In fact, $C_1$ is a strongly invertible L--space knot $v1980$.)
From this braid word, we make
a strongly invertible diagram as shown in Fig.~\ref{fig:v1980}. Then we can carry out the Montesinos trick.

\begin{figure}[th]
\centerline{\includegraphics[bb=0 0 646 222, width=12cm]{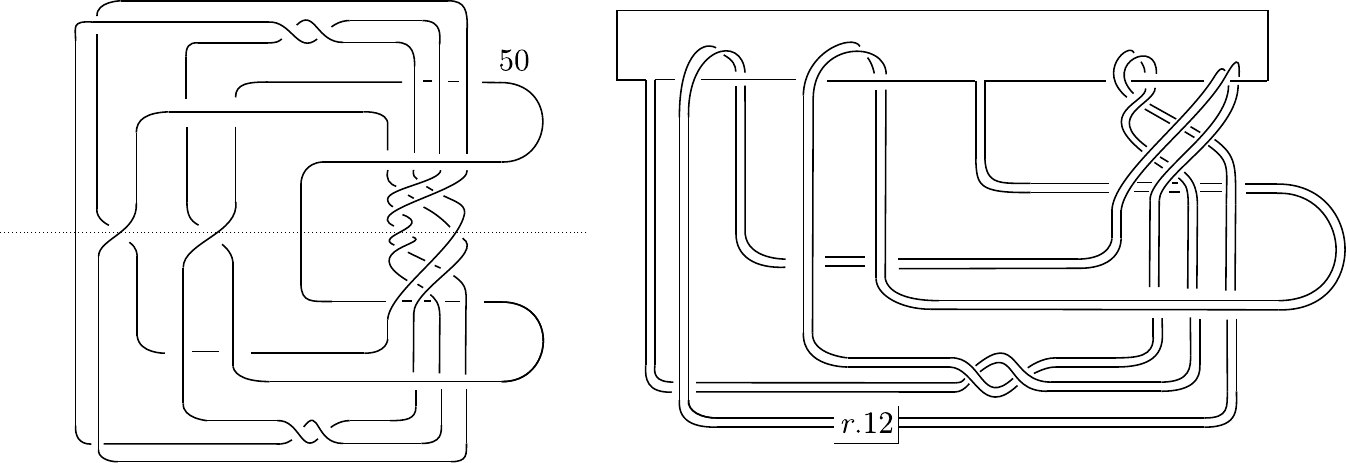}}
\vspace*{8pt}
\caption{Left: A strongly invertible diagram obtained from the braid word.
Right:  The link after the tangle replacement.
The box labeled with $r.12$ contains 12 right handed horizontal half twists.
\label{fig:v1980}}
\end{figure}

Take the quotient of the strongly invertible diagram.
After the tangle replacement, we have the link as shown in Fig.~\ref{fig:v1980}.


A series of deformations in Fig.~\ref{fig:L11n178} shows that this link and the mirror of $L11n178$ are
equivalent by observing how one component wraps around the other.
(Both components are unknotted.)
\end{proof}

\begin{figure}[th]
\centerline{\includegraphics[bb=0 0 738 782, width=12cm]{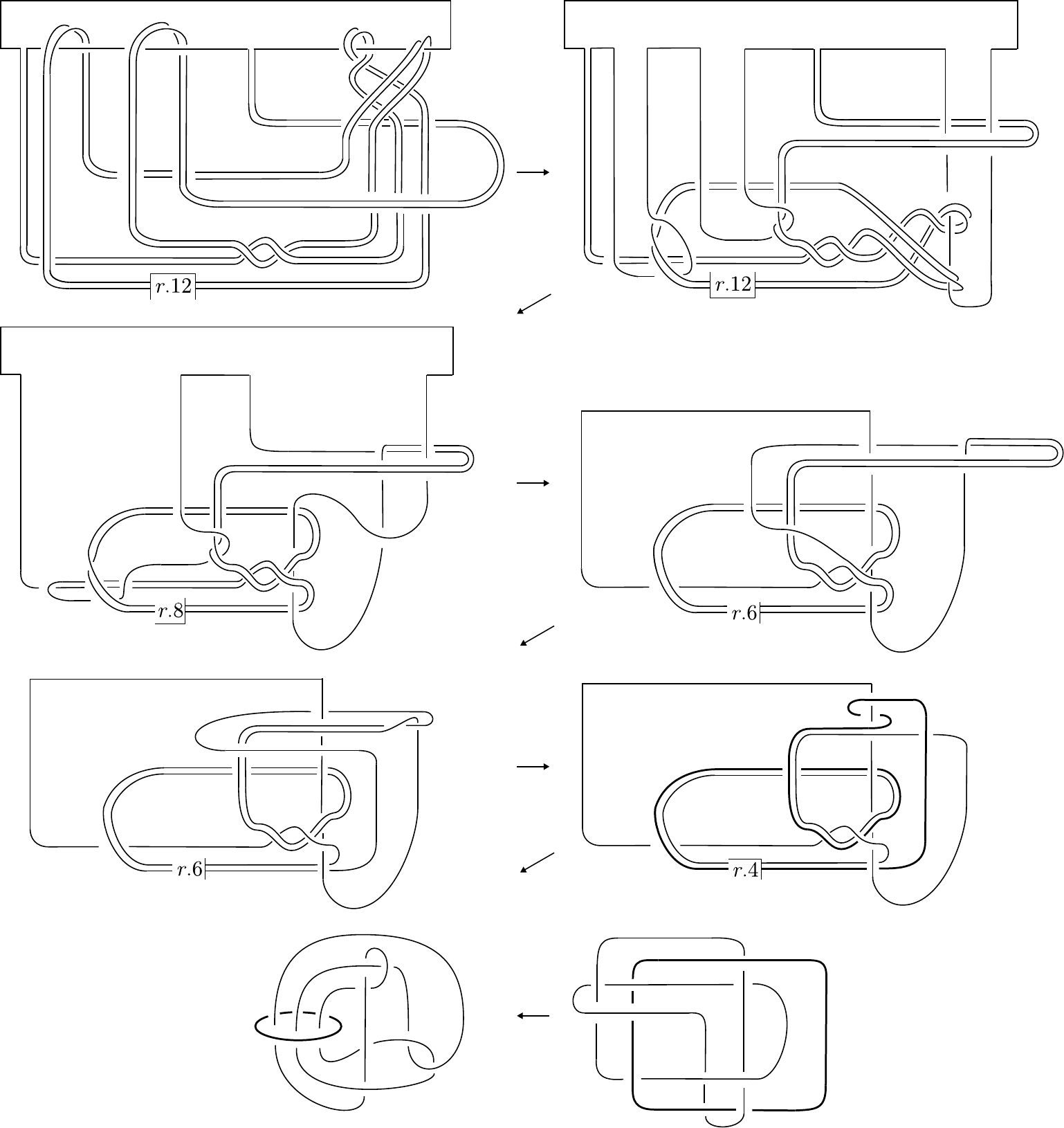}}
\vspace*{8pt}
\caption{The link after the tangle replacement is equivalent to the mirror of $L11n178$ (Bottom Right).
\label{fig:L11n178}}
\end{figure}

\section{$o9\_39162$}

\subsection{$64$--surgery}


\begin{lemma}
$L12n1968(\frac{1}{4},5,1)$ represents $(-64)$--surgery on the mirror of $o9\_39162$,
\end{lemma}

\begin{proof}
Do $-4$ twists on $C_0$, and then $-1$ twist on $C_1$.
See  Fig.~\ref{fig:o9_39162-64} and \ref{fig:o9_39162-64-2}.
The remaining component $C_2$ has coefficient $-64$ and a closed braid form of $6$ strands.
As shown in Fig.~\ref{fig:o9_39162-64-2}, 
it is the mirror of the closure of a $6$--braid
\[
\beta=[(1,2,3,4,5)^6,1,2,2,4,4,4,4,3,4,2,3,4,5,5,4,3,2,1,1,2,3,4,5].
\]
Set 
$\alpha=[4,4,1,2,2,3,2,4,3,1,2]$.
Then
\[
\alpha^{-1}\beta \alpha=[(1,2,3,4,5)^{12},-2,-1,-1,4,5,-3,-2,-1,-4,-3,-2].
\]
By renaming the generators $\sigma_i$ by $\sigma_{6-i}$, this gives
\begin{equation}\label{eq:o9_39162-64}
[(5,4,3,2,1)^{12},-4,-5,-5,2,1,-3,-4,-5,-2,-3,-4].
\end{equation}

\begin{figure}[th]
\centerline{\includegraphics[bb=0 0 390 261, width=10cm]{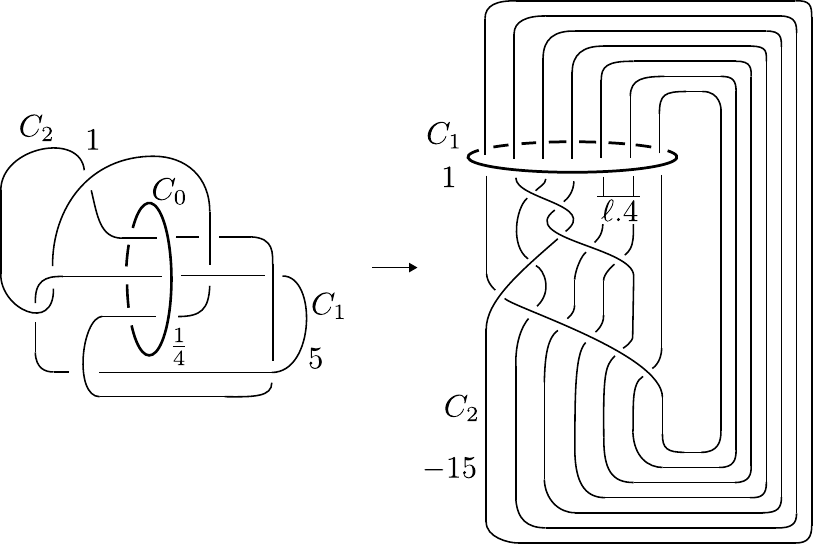}}
\vspace*{8pt}
\caption{Do $-4$ twists on $C_0$.
\label{fig:o9_39162-64}}
\end{figure}

\begin{figure}[th]
\centerline{\includegraphics[bb=0 0 607 261, width=12cm]{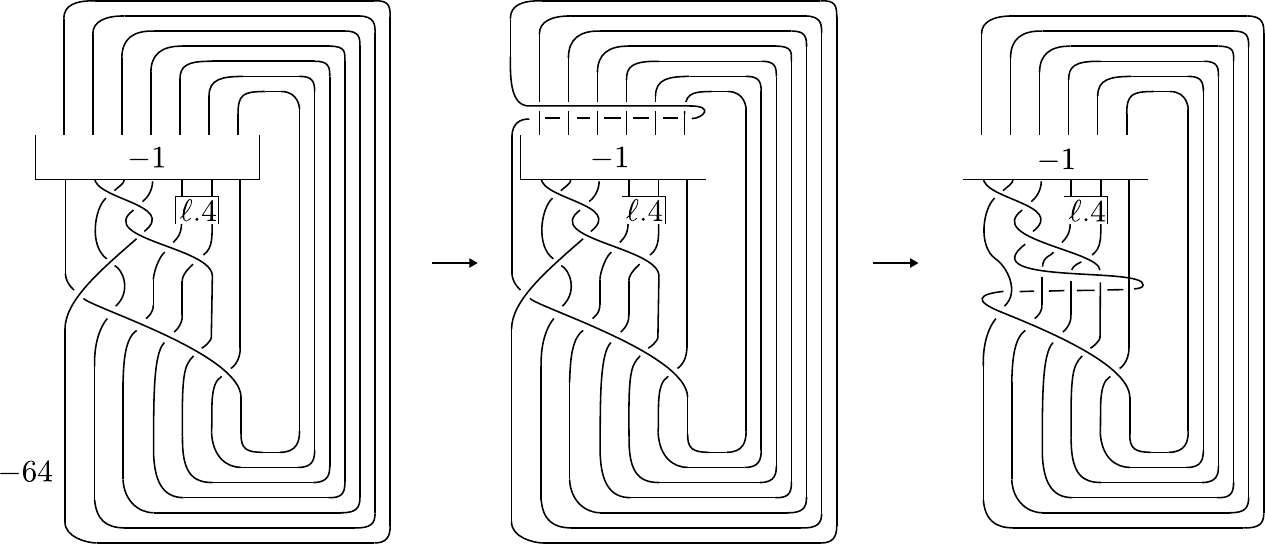}}
\vspace*{8pt}
\caption{After $-1$ twist on $C_1$, $C_2$ will be a closed braid form with $6$ strands.
\label{fig:o9_39162-64-2}}
\end{figure}

We will verify that its closure is  $o9\_39162$.
In Table \ref{table:braid}, $o9\_39162$ is
\[
\begin{split}
\gamma&=[1, 1, 2, 1, 3, 2, 4, 2, 5, 1, 3, 2, 2, 3, 2, 4, 2, 5, 2, 4, 3, 2, 4, 2, 5, 4, \\
&\qquad 3, 5, 2, 4, 2, 5, 2, 4, 3, 2, 2, 3, 3, 2, 2, 1, 2, 3, 4, 3, 2, 4, 5, 4, 3, 4, 3].
\end{split}
\]
Set 
\[
\begin{split}
\alpha'&=
[-4,-3,-2,-1,-1,2,3,4,1,2,3,4,2,3,2,3,-5,-4,\\
&\quad 3,2,1,-5,-4,-5,-3,2,-4,-3].
\end{split}
\]

Then we can see
\[
\alpha'^{-1}\gamma \alpha'=[(1,2,3,4,5)^{12},-4,-5,-5,2,1,-3,-4,-5,-2,-3,-4],
\]
which is equal to (\ref{eq:o9_39162-64}), because $(1,2,3,4,5)^{12}\ (=(5,4,3,2,1)^{12})$ is $2$ full twists.
\end{proof}

\begin{theorem}
The diagram $L12n1968(\frac{1}{4},5,1)$ represents the double branched cover of $L12n1050$.
\end{theorem}

\begin{proof}
As in the proof of Theorem \ref{lem:t12681dbc}, performing $-1$ twist on $C_2$
yields a pretzel link $P(2,4,-3)$.
(This corresponds to $L9n14(4,-\frac{15}{4})$ in SnapPy.)

For the strongly invertible diagram in Fig.~\ref{fig:L9n14} with different coefficients $(-15/4,4)$,
take the quotient.  
It is the same as Fig.~\ref{fig:L9n14} with $A$ and $B$ replaced by $[0,1,-3]$ and $4$, respectively.



We need to identify that this link is $L12n1050$.
As shown in Fig.~\ref{fig:L12n1050}, we examine how one component wraps around
the other.
This confirms that both links are the same. 
\end{proof}

\begin{figure}[th]
\centerline{\includegraphics[bb=0 0 487 103, width=12cm]{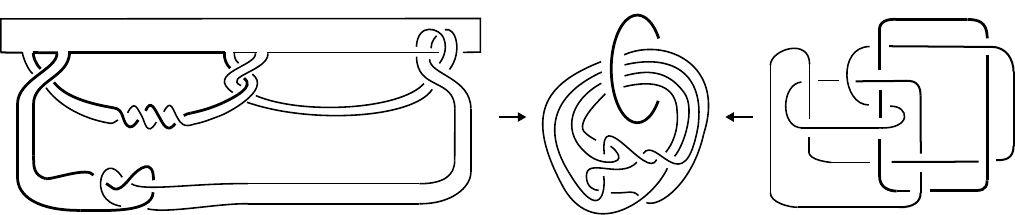}}
\vspace*{8pt}
\caption{Right is $L12n1050$.
\label{fig:L12n1050}}
\end{figure}

\subsection{$65$--surgery}

\begin{lemma}
$L13n9366(1,\frac{3}{4},5)$ represents
$65$--surgery on $o9\_39162$.
\end{lemma}

\begin{proof}
Perform $-1$ twist on $C_1$.
Then $C_1$ and $C_2$ have coefficients $3$ and $1$, respectively (Fig~\ref{fig:o9_39162-65}).
Do $-1$ twist on $C_2$ to erase it.
Finally, do $+1$ twist on $C_1$.  Then $C_0$ gives a knot with coefficient $65$, which is
of a closed braid form with $9$ strands.

\begin{figure}[th]
\centerline{\includegraphics[bb=0 0 532 154, width=12cm]{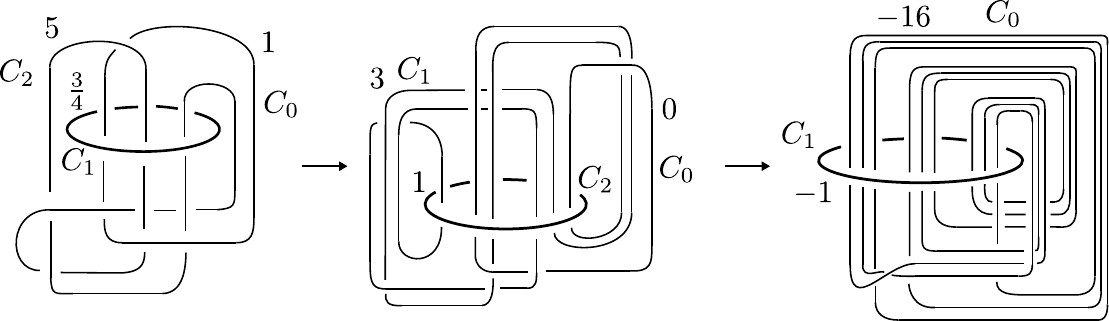}}
\vspace*{8pt}
\caption{Do $-1$ twist on $C_0$, and then $-1$ twist on $C_2$.
\label{fig:o9_39162-65}}
\end{figure}

First, we reduce $C_0$ into a closure of a $6$--braid $\beta$ as shown in Fig.~\ref{fig:o9_39162-65-2}, where
\[
\beta=[(1,2,3,4,5)^6,4,3,2,1,(1,2,3,4,5)^2,5,1,2,3,4,5,5,4,3].
\]

\begin{figure}[th]
\centerline{\includegraphics[bb=0 0 464 464, width=12cm]{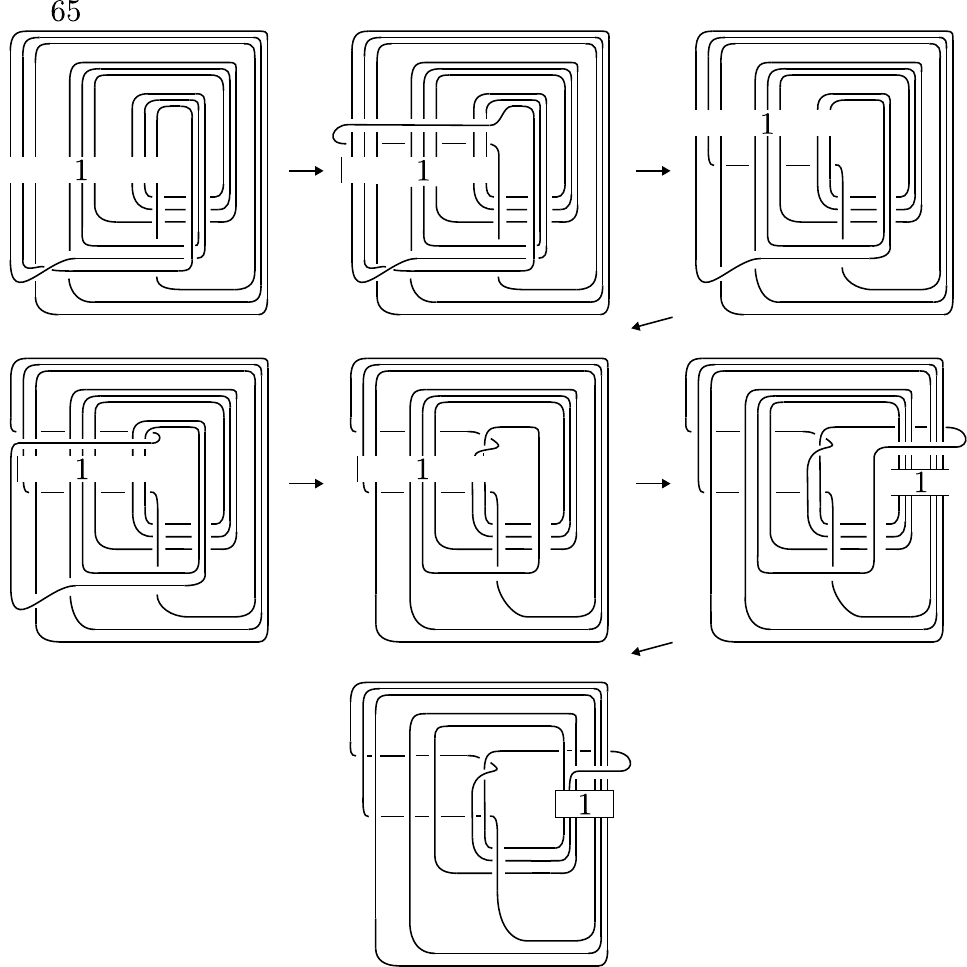}}
\vspace*{8pt}
\caption{Deform $C_0$ into the closure of a $6$--braid.
\label{fig:o9_39162-65-2}}
\end{figure}

Set $\alpha=[4,1,3,2,1,-5,-4,2,1,2]$.
Then we have
\[
\alpha^{-1}\beta \alpha=[(1,2,3,4,5)^{12},-4,-3,-2,-5,-4,-3,1,2,-5,-5,-4].
\]
By reversing the order of words, this is conjugate to (\ref{eq:o9_39162-64}).
\end{proof}


\begin{theorem}
The diagram $L13n9366(1,\frac{3}{4},5)$ represents the double branched cover of $K12n278$.
\end{theorem}

\begin{proof}
In the diagram $L13n9366(1,\frac{3}{4},5)$, do $-1$ twist on $C_0$.
This yields a strongly invertible diagram of a $2$--bridge link, which is $L8a12$.
See Fig.~\ref{fig:L8a12}.

\begin{figure}[th]
\centerline{\includegraphics[bb=0 0 307 104, width=9cm]{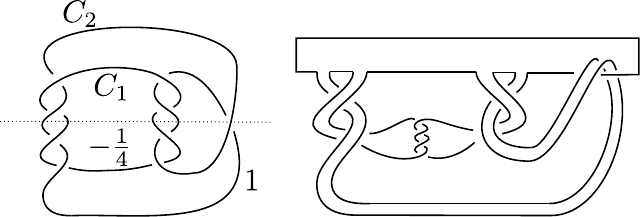}}
\vspace*{8pt}
\caption{Left: A strongly invertible position after $-1$ twist on $C_0$ in the diagram $L13n9366(1,\frac{3}{4},5)$.  Right: The link after tangle replacement.
\label{fig:L8a12}}
\end{figure}

Take the quotient.  Then it is not hard to see that the tangle replacement yields the Montesinos knot $(\frac{2}{3},\frac{2}{5},-\frac{1}{5})$, which is the mirror of $K12n278$ (Fig.~\ref{fig:L8a12}).
\end{proof}

\section{$o9\_40363$}

\subsection{$82$--surgery}

We use $L14n63000$ again (Fig.~\ref{fig:start}).
The only difference from the case of $t12681$ in Section \ref{sec:t12681} is the coefficient on $C_0$.

\begin{lemma}
$L14n63000(\frac{3}{4},1,\frac{1}{2},1)$ represents $(-82)$--surgery on the mirror of $o9\_40363$.
\end{lemma}

\begin{proof}
As in the proof of Lemma \ref{lem:t12681-62}, 
do $-1$ twist on $C_3$, and then $4$ twists on $C_0$.
See Fig.~\ref{fig:o9_40363-82}.

\begin{figure}[th]
\centerline{\includegraphics[bb=0 0 542 139, width=12cm]{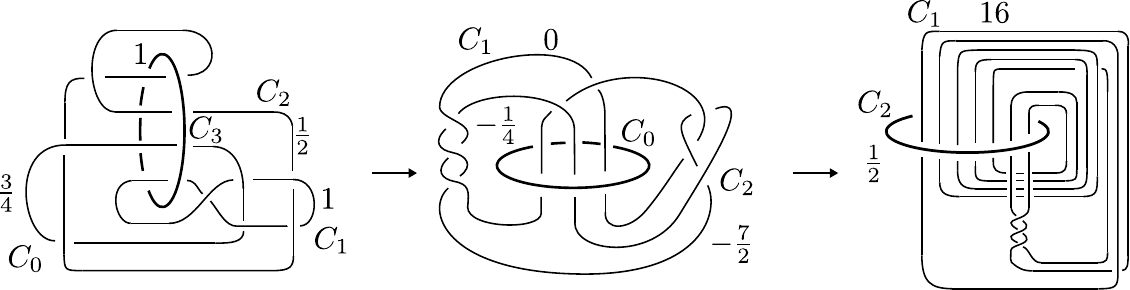}}
\vspace*{8pt}
\caption{Do $-1$ twist on $C_3$, and then $4$ twists on $C_0$.
\label{fig:o9_40363-82}}
\end{figure}

Finally, $-2$ twists on $C_2$ changes $C_1$ into a knot with coefficient $-82$, which is
of a closed braid form with $7$ strands.
It is the closure of a braid
\[
\beta=[(1,2,3,4,5,6)^{-14},5,4,3,2,6,5,4,3,2,2,2,-1,3,4],
\]
where $(1,2,3,4,5,6)^{-14}$ corresponds to $-2$ full twists.
Then for $\alpha=[-4,-6]$,
\begin{equation}\label{eq:o9_40363-82}
\alpha^{-1}\beta \alpha=
[(1,2,3,4,5,6)^{-14},5,4,3,2,6,5,4,3,2,2,2,2,-1,3].
\end{equation}

We verify that the closure of $\beta$ is the mirror of $o9\_40363$. 
From Table \ref{table:braid},  $o9\_40363$ is the closure of a braid
\[
\begin{split}
\gamma&=[1, 2, 1, 3, 4, 5, 4, 4, 4, 5, 4, 6, 3, 6, 2, 5, 4, 3, 5, 4, 6, 5, 6, 6, 5, 4, 6, 3, 5, 2, 4, 1, \\
& 3, 5, 2, 4, 6, 3, 5, 4, 6, 5, 6, 6, 5, 4, 6, 3, 5, 2, 4, 6, 1, 3, 5, 2, 4, 6, 3, 5, 4, 6, 5, 4, 6,\\
& 3, 5, 2, 4, 3, 5, 4].
\end{split}
\]
Set $\alpha=[1,2,3,5,4,5,-2,-1,6,5,4,6,5,6]$.  Then
\[
\alpha^{-1}\gamma \alpha=[-4,6,(-5)^{4},-4,-3,-2,-1,-5,-4,-3,-2,(6,5,4,3,2,1)^{14}].
\]
If we rename the generators $\sigma_i$ with $\sigma_{7-i}$, then
this is
\[
[-3,1,(-2)^4,-3,-4,-5,-6,-2,-3,-4,-5,(1,2,3,4,5,6)^{14}],
\]
whose closure is the mirror of (\ref{eq:o9_40363-82}).
\end{proof}

\begin{theorem}
The diagram $L14n63000(\frac{3}{4},1,\frac{1}{2},1)$ represents the double branched cover of $L12n785$.
\end{theorem}

\begin{proof}
The only difference with Theorem \ref{lem:t12681-62} is the coefficient of $C_0$.
Hence we replace the tangle $A$ in Fig.~\ref{fig:L10n92q} with $-\frac{1}{4}=[0,4]$.

As in Fig.~\ref{fig:L10n92q}, we modify the link after the tangle replacement.
By examining how the knotted component, which is the knot $6_1$, wraps around the unknotted component,
we can see that  the resulting link is the mirror of $L12n785$  as shown in Fig.~\ref{fig:L12n785}.
\end{proof}

\begin{figure}[th]
\centerline{\includegraphics[bb=0 0 482 113, width=12cm]{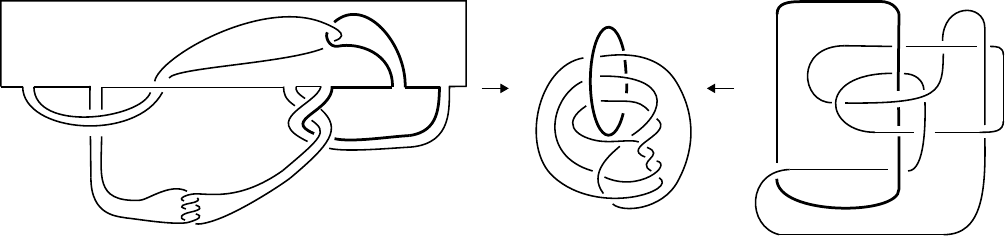}}
\vspace*{8pt}
\caption{The link after tangle replacement is the mirror of $L12n785$ (Right).
\label{fig:L12n785}}
\end{figure}

\subsection{$83$--surgery}

\begin{lemma}
$L12n1968(1,\frac{7}{2},\frac{1}{4})$ represents $83$--surgery on $o9\_40363$.
\end{lemma}

\begin{proof}
Do $-4$ twists on $C_2$ to erase it (Fig.~\ref{fig:o9_40363-83}).
Then $2$ twists on $C_1$ changes $C_0$ to a knot with coefficient $83$, which 
is of a closed braid form of $7$ strands.

\begin{figure}[th]
\centerline{\includegraphics[bb=0 0 400 190, width=9cm]{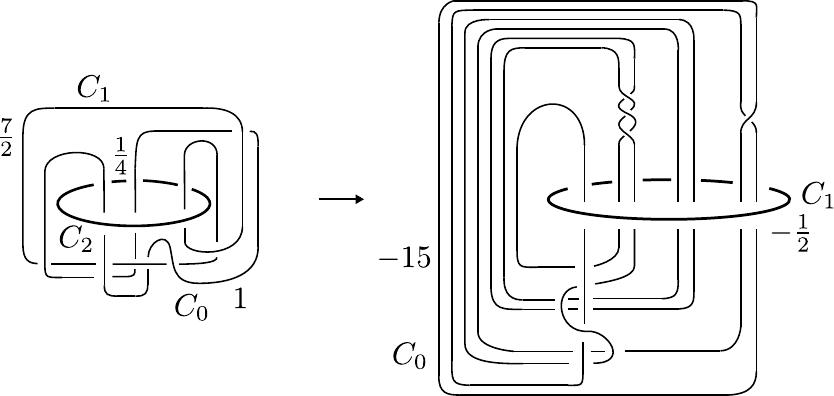}}
\vspace*{8pt}
\caption{Do $-4$ twists on $C_2$, and then $2$ twists on $C_1$.
\label{fig:o9_40363-83}}
\end{figure}

Let 
\[
\beta=[(1,2,3,4,5,6)^{14}, 1,(-5)^3,-2,-3,-2,-3,-4,-5,-3,-4,-5,-6].
\]
Then the knot is the closure of $\beta$.
Taking a conjugation with 
$\alpha=[1,-2,-3,-2,-3,1,6,5]$,
we have
\[
\alpha^{-1}\beta\alpha=[(1,2,3,4,5,6)^{14},-5,-4,-3,-2,-6,-5,-4,-3,(-2)^4,1,-3],
\]
which is the mirror of (\ref{eq:o9_40363-82}).
\end{proof}

\begin{theorem}
The diagram $L12n1968(1,\frac{7}{2},\frac{1}{4})$ represents the double branched cover of $K12n479$.
\end{theorem}

\begin{proof}
Perform $-1$ twist on $C_0$.  Then  the result is the pretzel link $L9n14=P(2,4,-3)$, which is the same as in the proof of Theorem \ref{lem:t12681dbc}.
(In the link $L12n1968$, the first and third components are interchangeable.)
Thus the diagram is equal to that of Fig.~\ref{fig:L9n14} where $C_1$ (resp. $C_2$) has coefficient $\frac{5}{2}$ (resp. $-\frac{15}{4}$).

Take the quotient, which is Fig.~\ref{fig:L9n14} where the tangles $A$ and $B$ are replaced with $[0,1,-3]$ and $[2,-2]$, respectively.
Then it is easy to see that the result is the Montesinos knot $(\frac{4}{7},\frac{3}{4},-\frac{1}{3})$, which is  $K12n479$.
\end{proof}

\section{$o9\_40487$}

\subsection{$37$--surgery}

\begin{lemma}
$L13n8037(-\frac{1}{2},-1,-3)$ represents $(-37)$--surgery on the mirror of $o9\_40487$.
\end{lemma}

\begin{proof}

Performing $2$ twists on $C_0$ and then $-1$ twist on $C_1$ change
$C_2$ into a knot with coefficient $-37$ as the  closure of a $6$--braid (Fig.~\ref{fig:o9_40487-37}).

\begin{figure}[th]
\centerline{\includegraphics[bb=0 0 380 151, width=10cm]{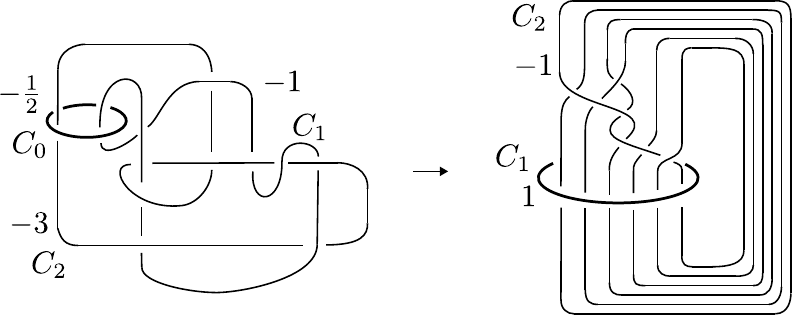}}
\vspace*{8pt}
\caption{Do $2$ twists on $C_0$, and then $-1$ twist on $C_1$.
\label{fig:o9_40487-37}}
\end{figure}

Moreover, it is deformed to the closure of a $5$--braid.
Let
\[
\beta=[(1,2,3,4)^{5},1,2,3,4,4,3,2,1,1,2,2,3,4,-2].
\]
Then our knot is the mirror of the closure of $\beta$ (Fig.~\ref{fig:o9_40487-37-2}).

\begin{figure}[th]
\centerline{\includegraphics[bb=0 0 428 164, width=10cm]{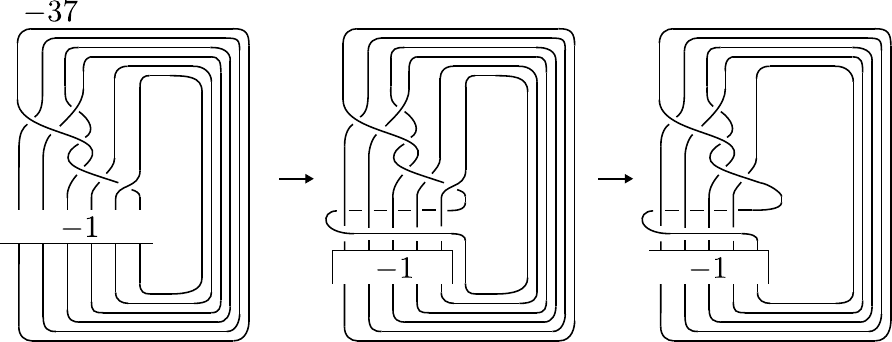}}
\vspace*{8pt}
\caption{The knot is deformed into the closure of a $5$--braid.
\label{fig:o9_40487-37-2}}
\end{figure}

By taking a conjugation with $\alpha=[1,2,3,1,4]$,
\begin{equation}\label{eq:o9_40487-37}
\alpha^{-1}\beta \alpha=[(1,2,3,4)^5,4,3,2,1,1,2,2,1,3,2,4,3].
\end{equation}

We will verify that this is equal to $o9\_40487$.
From Table \ref{table:braid}, $o9\_40487$ is
\[
\gamma=[1, 2, 1, 3, 3, 2, 2, 3, 4, 3, 2, 1, 3, 2, 1, 3, 2, 4, 2, 4, 1, 4, 2, 1, 3, 2, 3, 4, 3, 4, 3, 2].
\]
Let $\alpha'=[-3,1,2,3]$.
Then
\begin{equation}\label{eq:o9_40487-37-2}
\alpha'^{-1}\gamma \alpha'=[(1,2,3,4)^{5},2,1,3,2,4,3,3,4,4,3,2,1].
\end{equation}
By renaming the generators $\sigma_i$ with $\sigma_{5-i}$, this yields
\[
[(4,3,2,1)^5, 3,4,2,3,1,2,2,1,1,2,3,4]
\]
Reversing the order of words, we can see that this is (\ref{eq:o9_40487-37}).
\end{proof}

\begin{theorem}
The diagram $L13n8037(-\frac{1}{2},-1,-3)$ represents the double branched cover of $K11n147$.
\end{theorem}

\begin{proof}
In the diagram $L13n8037(-\frac{1}{2},-1,-3)$,
do $2$ twists on $C_0$ to erase it.
The resulting link consisting of $C_1\cup C_2$ is interchangeable.
(It is the Montesinos link  $(\frac{1}{5},-\frac{2}{3},\frac{3}{7})$.)
Here, $C_1$  and $C_2$ have coefficients $1$ and $-1$, respectively.
Since this link is not strongly invertible, perform further $+1$ twist on $C_2$.
Then we have a strongly invertible knot with coefficient $37$, which is the mirror of  $m239$ in SnapPy.

Figure \ref{fig:m239} shows its diagram.
After taking the quotient, the tangle replacement is done by $12$, because the diagram has writhe $25$.
See Fig.~\ref{fig:m239}.

\begin{figure}[th]
\centerline{\includegraphics[bb=0 0 529 185, width=11cm]{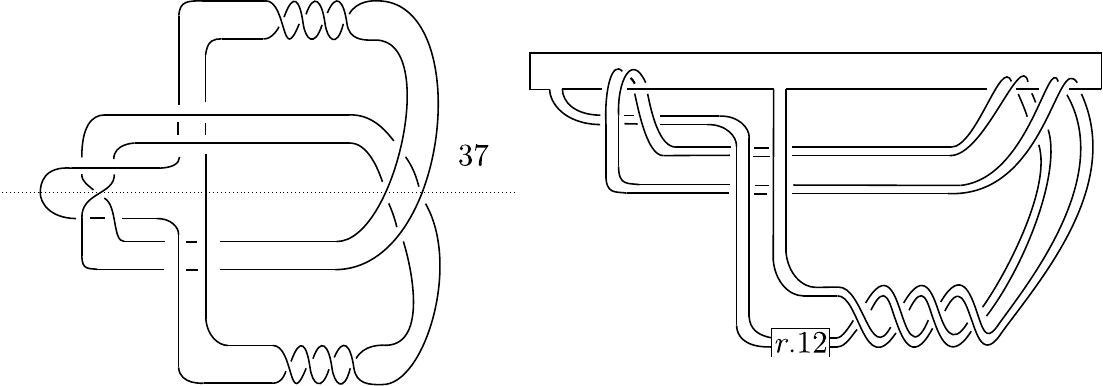}}
\vspace*{8pt}
\caption{Left: A strongly invertible diagram.
Right: The knot after the tangle replacement.
The box contains $12$ right half twists.
\label{fig:m239}}
\end{figure}


We need to identify this as $K11n147$.
As shown in Fig.~\ref{fig:K11n147},
the knot in Fig.~\ref{fig:m239} is expressed as the closure of a braid after some deformation.
Then it is not hard to confirm that this is equivalent to $K11n147$.
\end{proof}

\begin{figure}[th]
\centerline{\includegraphics[bb=0 0 574 460, width=12cm]{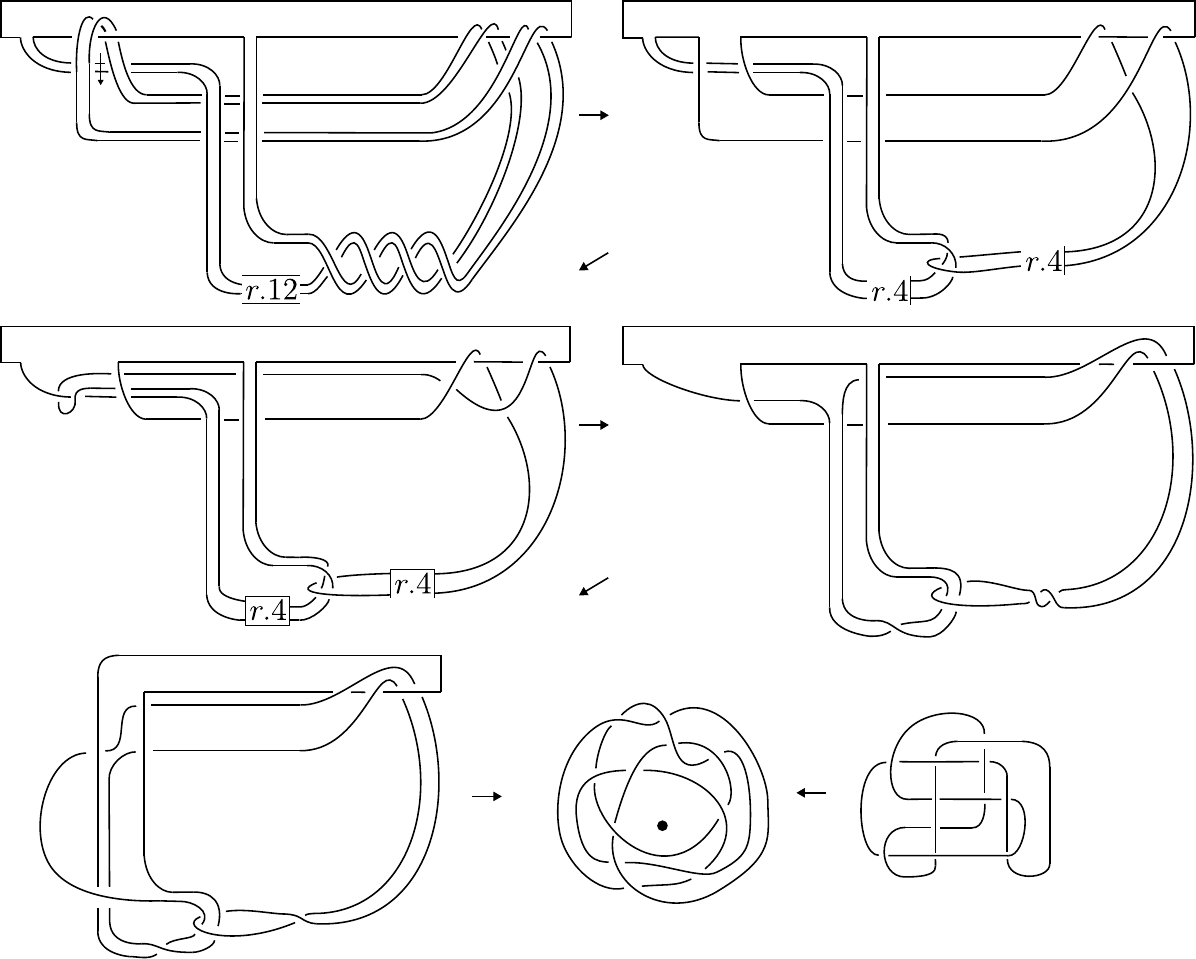}}
\vspace*{8pt}
\caption{The equivalence with $K11n147$ (Bottom Right).
\label{fig:K11n147}}
\end{figure}


\subsection{$38$--surgery}

\begin{lemma}
$L12n1625(-\frac{1}{3},-4,-1)$ represents
$38$--surgery on $o9\_40487$.
\end{lemma}

\begin{proof}
Do $3$ twists on $C_0$, and then perform $+1$ twist on $C_1$.
Then $C_2$ will be a knot with coefficient $38$ as the closure of a $6$--braid.
See Fig.~\ref{fig:o9_40487-38}.

\begin{figure}[th]
\centerline{\includegraphics[bb=0 0 298 151, width=8cm]{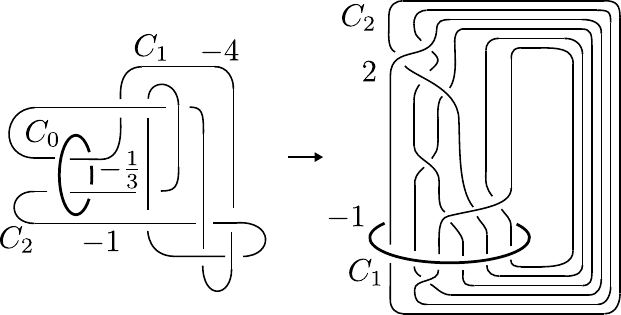}}
\vspace*{8pt}
\caption{Do $3$ twists on$C_0$, and then $+1$ twist on $C_1$ changes
$C_2$ into a knot with coefficient $38$ in a closed braid form of $6$ strands.
\label{fig:o9_40487-38}}
\end{figure}

We can deform this into the closure of a $5$--braid $\beta$ as illustrated in Fig.~\ref{fig:o9_40487-38-2},
where
\[
\beta=[(1,2,3,4)^5,3,2,1,-3,4,3,3,4,4,3,2,1,1,2].
\]
Set $\alpha=[-2,-1,-2]$.
Then 
\[
\alpha^{-1}\beta \alpha=[(1,2,3,4)^5,2,1,3,2,4,3,3,4,4,3,2,1],
\]
which is equal to (\ref{eq:o9_40487-37-2}).
Hence, this knot is $o9\_40487$ as desired. 
\end{proof}

\begin{figure}[th]
\centerline{\includegraphics[bb=0 0 412 385, width=11cm]{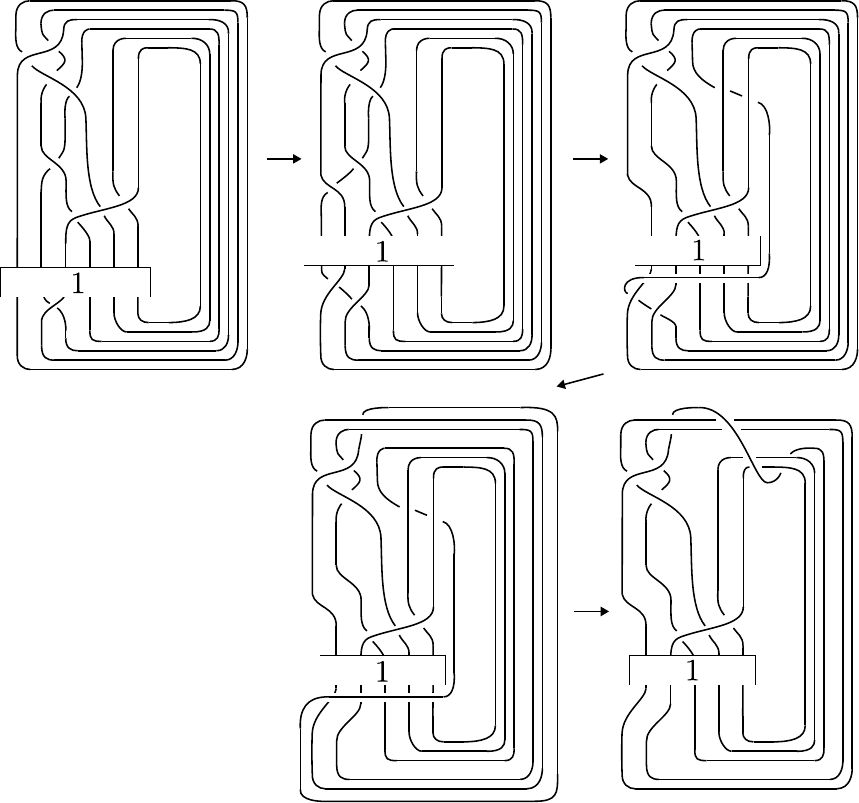}}
\vspace*{8pt}
\caption{Deform into a $5$--braid.
\label{fig:o9_40487-38-2}}
\end{figure}


\begin{theorem}\label{lem:o9_40487-38}
The diagram $L12n1625(-\frac{1}{3},-4,-1)$ represents the double branched cover of $L11n152$.
\end{theorem}

\begin{proof}
Performing $+1$ twist on $C_2$ yields a strongly invertible link $C_0\cup C_1$ as shown in Fig.~\ref{fig:L13n4344}.

\begin{figure}[th]
\centerline{\includegraphics[bb=0 0 494 158, width=12cm]{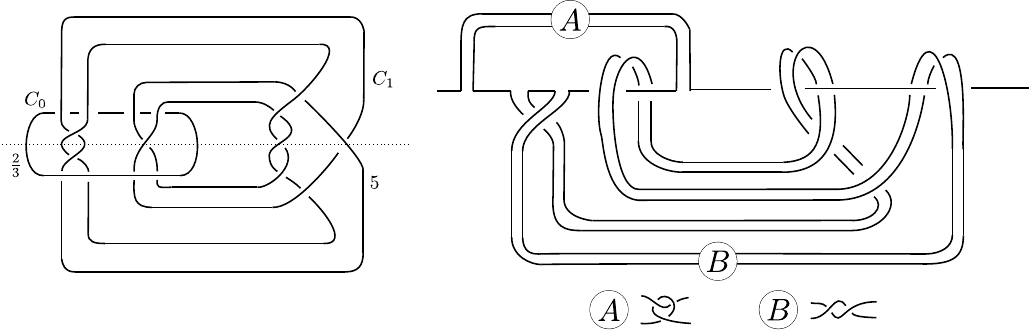}}
\vspace*{8pt}
\caption{Left: A strongly invertible diagram.
Right: The link after the tangle replacement.
\label{fig:L13n4344}}
\end{figure}

After taking the quotient (Fig.~\ref{fig:L13n4344}), we can easily see that there is the unknotted component as shown in Fig.~\ref{fig:L11n152eq}.
Then by examining how the other component wraps around the unknotted one,
we can confirm that this link is $L11n152$.  
\end{proof}


\begin{figure}[th]
\centerline{\includegraphics[bb=0 0 706 162, width=12cm]{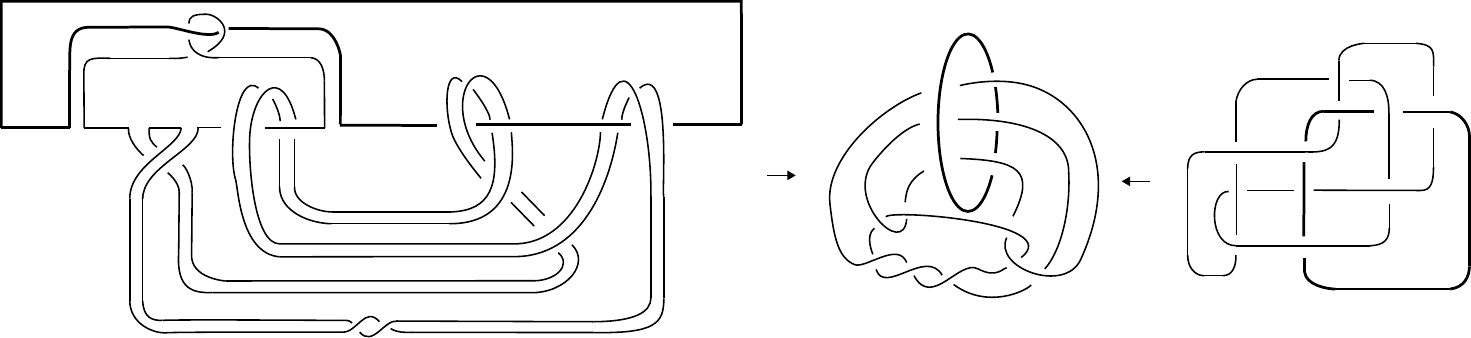}}
\vspace*{8pt}
\caption{The link in Fig.~\ref{fig:L13n4344} is equivalent to $L11n152$ (Right).
\label{fig:L11n152eq}}
\end{figure}

\section{$o9\_40504$}

\subsection{$58$--surgery}

\begin{lemma}
$L14n62791(-\frac{1}{3},-2,\frac{1}{2},-1)$ represents
$58$--surgery on $o9\_40504$.
\end{lemma}

\begin{proof}
Do $3$ twists on $C_0$, and then $-1$ twist on $C_1$.
Then $2$ twists on $C_2$ change $C_3$ to a knot  with coefficient $58$.
See Fig.~\ref{fig:o9_40504-58}.

\begin{figure}[th]
\centerline{\includegraphics[bb=0 0 481 192, width=12cm]{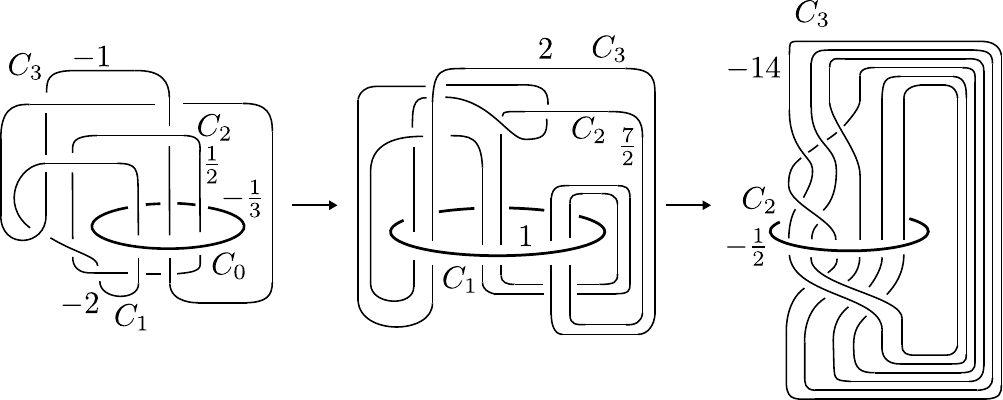}}
\vspace*{8pt}
\caption{Do $3$ twists on $C_0$, and then $-1$ twist on $C_1$.
\label{fig:o9_40504-58}}
\end{figure}

Let
\[
\beta=[(1,2,3,4,5)^{12},-2,-3,-4,-5,-1,-2,-3,-4,-4-5,-5,-4,-3]).
\]
Then our knot is the closure of $\beta$.
By spending $+1$ full twist $(1,2,3,4,5)^{6}$, we cancel the negative elements.
Set $\alpha=[3,4,5,5,4,3]$.
Then
\begin{equation}\label{eq:o9_40504-58}
\alpha^{-1}\beta \alpha=[(1,2,3,4,5)^{8}, 2,3,2,2,2,3,2].
\end{equation}

On the other hand, $o9\_40504$ is
\[
\begin{split}
\gamma&=[1, 1, 2, 1, 3, 4, 3, 4, 3, 5, 4, 3, 5, 2, 4, 1, 3, 1, 2, 1, 3, 4, 5, 4, 3, 5, 4, 3, 5, 5, 5, 4, \\
&\qquad 3, 2, 1, 3, 4, 5, 4, 4, 5, 4, 3, 2, 4, 3, 4]
\end{split}
\]
from Table \ref{table:braid}.
Set $\alpha'=[-4,-3,-2,-1,5,2,4,3,-5,-4,-3,5,4]$.  Then we see
\begin{equation}\label{eq:o9_40504-58word}
\alpha'^{-1}\gamma \alpha'=[(1,2,3,4,5)^{8},5,4,5,5,5,4,5].
\end{equation}
By renaming the generators $\sigma_i$ with $\sigma_{6-i}$,  this gives
\[
[(5,4,3,2,1)^8, 1,2,1,1,1,2,1].
\]
Further,
taking a conjugation with $(5,4,3,2,1)$,
\[
[(5,4,3,2,1)^7,1,2,1,1,1,2,1,(5,4,3,2,1)]=[(5,4,3,2,1)^8,2,3,2,2,2,3,2].
\]
By reversing the order of words, we can conclude from (\ref{eq:o9_40504-58})
that the closure of $\beta$ is  $o9\_40504$
\end{proof}

\begin{theorem}
The diagram $L14n62791(-\frac{1}{3},-2,\frac{1}{2},-1)$ represents the double branched cover of $L11n179$.
\end{theorem}

\begin{proof}
Performing $+1$ twist on $C_3$ yields a strongly invertible link as shown in Fig.~\ref{fig:L10n88},
which is the mirror of $L10n88$.

After taking the quotient, we can see that the tangle replacement yields the mirror of $L11n179$ (Fig.~\ref{fig:L10n88}).
As shown in Fig.~\ref{fig:L11n179eq}, we can easily see the unknotted component.
The other one is the figure eight knot.
Then we can see the equivalence by examining how the knotted component wraps around
the unknotted one as in Fig.~\ref{fig:L11n179eq}.
\end{proof}

\begin{figure}[th]
\centerline{\includegraphics[bb=0 0 505 213, width=12cm]{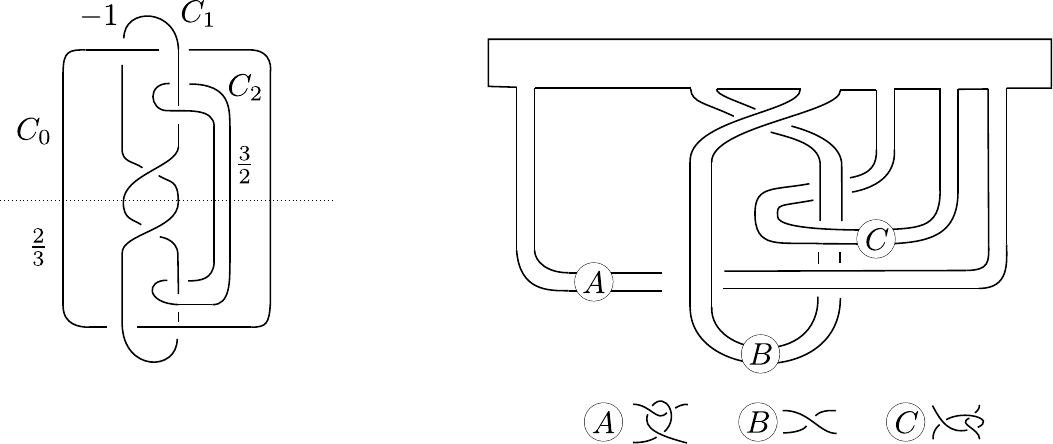}}
\vspace*{8pt}
\caption{Left: The mirror of $L10n88(-\frac{2}{3},1,-\frac{3}{2})$.
Right:  The link after the tangle replacement.
\label{fig:L10n88}}
\end{figure}


\begin{figure}[th]
\centerline{\includegraphics[bb=0 0 547 137, width=12cm]{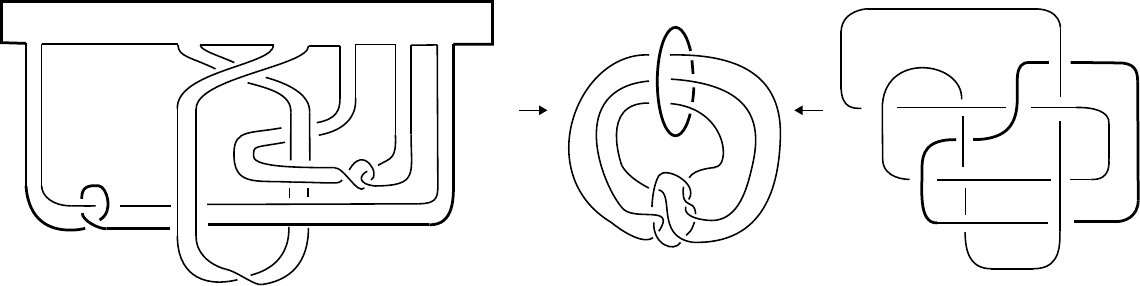}}
\vspace*{8pt}
\caption{An equivalence between the link in Fig.~\ref{fig:L10n88} and the mirror of $L11n179$ (Right).
\label{fig:L11n179eq}}
\end{figure}


\subsection{$59$--surgery}


\begin{lemma}
$L11n456(-\frac{3}{5},-2,-1,-2)$ represents $(-59)$--surgery on the mirror of $o9\_40504$.
\end{lemma}

\begin{proof}
The only difference from the case $o9\_38928$ (Subsection \ref{subsec:o9_38928-49})  is the coefficient on $C_0$.
Hence, just change the coefficient on $C_0$ to $-\frac{3}{5}$ in Fig.~\ref{fig:o9_38928-49}.

After $+1$ twist on $C_3$ and $+1$ twist on $C_1$, we have the diagram as in Fig.~\ref{fig:o9_38928-49-2} (Left), where the coefficient on $C_0$ is changed to $\frac{2}{5}$.
Perform $-2$ twist on $C_0$.
At this point, $C_0$, $C_2$, $C_3$ have coefficients $2$, $-14$, $1$, respectively (Fig.~\ref{fig:o9_40504-59}).
Do $-1$ twist on $C_3$ to erase it.
Finally, $-1$ twist on $C_0$ changes $C_2$ into a knot with coefficient $-59$.

\begin{figure}[th]
\centerline{\includegraphics[bb=0 0 362 193, width=9cm]{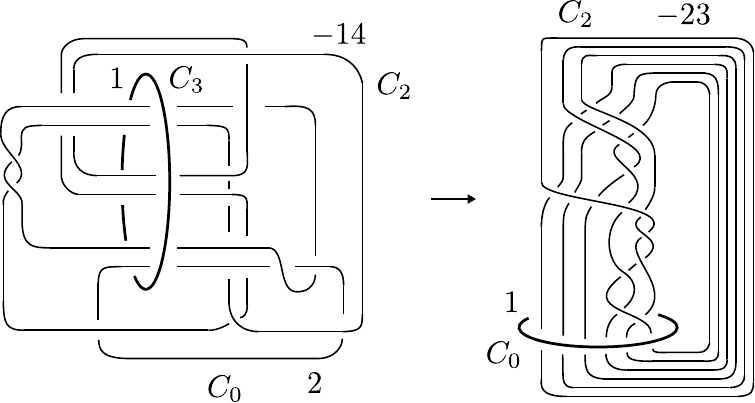}}
\vspace*{8pt}
\caption{Do $-1$ twist on $C_0$, and then $-1$ twist on $C_0$ yields a knot with
coefficient $-59$.
\label{fig:o9_40504-59}}
\end{figure}

Let 
\[
\beta=[(1,2,3,4,5)^6,3,2,4,3,5,4,4,1,2,3,4,5,5,5,4,4,5].
\]
Then our knot is the mirror of the closure of $\beta$.
We identify the closure of $\beta$ as $o9\_40504$.
Set $\alpha=[-5,-2,3,4,4,1,3]$.
Then
\begin{align*}
\alpha^{-1}\beta\alpha&=[(1,2,3,4,5)^8,4,3,4,4,4,3,4]\\
&=[(1,2,3,4,5)^7,5,4,5,5,5,4,5,(1,2,3,4,5)].
\end{align*}
By taking a conjugation with $(-5,-4,-3,-2,-1)$, we have
\[
[(1,2,3,4,5)^8,5,4,5,5,5,4,5],
\]
which is equal to (\ref{eq:o9_40504-58word}).
\end{proof}

\begin{theorem}
The diagram $L11n456(-\frac{3}{5},-2,-1,-2)$ represents the double branched cover of $K11n166$.
\end{theorem}

\begin{proof}
Perform $+1$ twist on $C_2$.
As in the proof of Theorem \ref{lem:o9_38928dbc},
$+1$ twist on $C_2$ yields a strongly invertible diagram in Fig.\ref{fig:L12n1896} with 
new coefficients $ (\frac{2}{5},-1,-1)$ for $(C_0,C_1,C_3)$.
Thus in Fig.~\ref{fig:L12n1896},
only the tangle $A$ is  replaced with $\frac{2}{5}=[0,-2,2]$.

Again, we can  identify  that the link after the tangle replacement is the mirror of $K11n166$.
Figure \ref{fig:K11n166} shows a series of deformations.
\end{proof}

\begin{figure}[th]
\centerline{\includegraphics[bb=0 0 542 461, width=12cm]{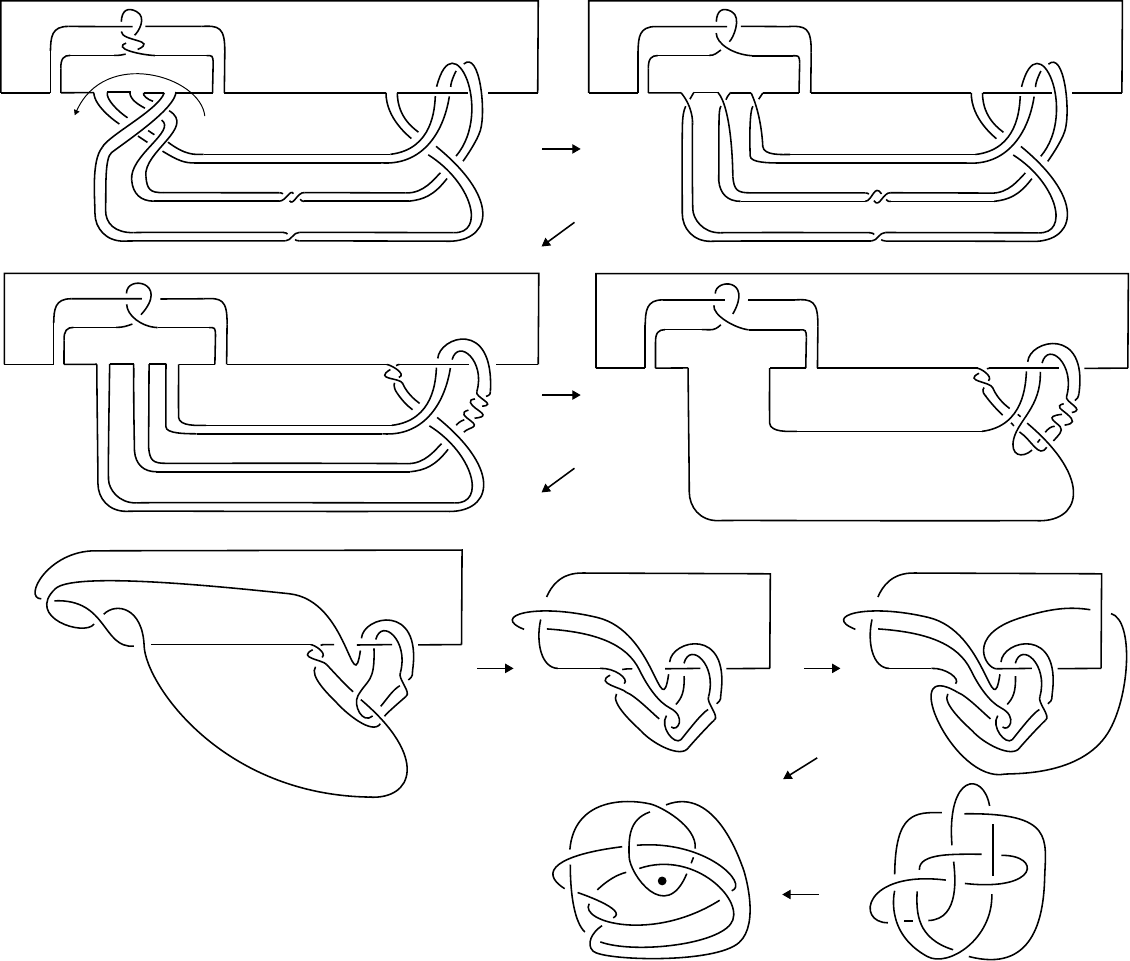}}
\vspace*{8pt}
\caption{An equivalence with the mirror of $K11n166$ (Bottom Right).
\label{fig:K11n166}}
\end{figure}

\section{$o9\_40582$}

\subsection{$43$--surgery}

\begin{lemma}
$L12n1638(-\frac{2}{3},1,-2)$ represents $43$--surgery on $o9\_40582$.
\end{lemma}

\begin{proof}
Do $+1$ twist on $C_0$, and then $+1$ twist on $C_2$ to erase it (Fig.~\ref{fig:o9_40582-43}).
Finally, $+1$ twist on $C_0$ changes $C_1$ into a knot with coefficient $43$, which
has a closed braid form of $5$ strands.
Let $\beta=[(1,2,3,4)^5,(1,2,3)^4,2,1,3,3,4,-3]$.

\begin{figure}[th]
\centerline{\includegraphics[bb=0 0 461 187, width=12cm]{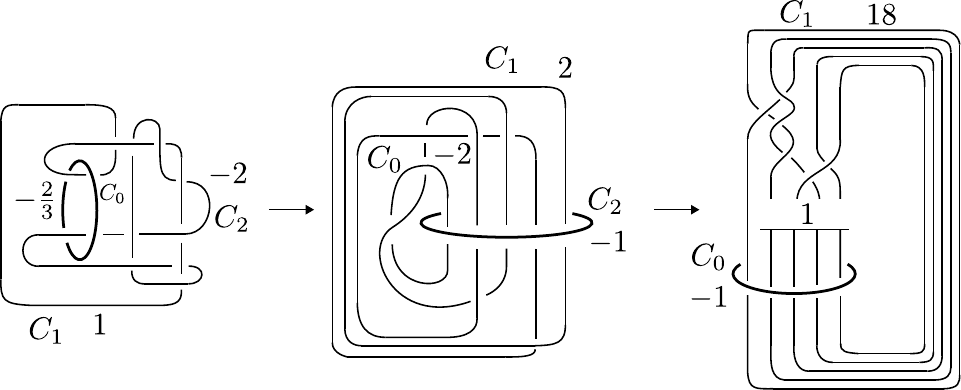}}
\vspace*{8pt}
\caption{Do $+1$ twist on $C_0$, and then $+1$ twist on $C_2$.
\label{fig:o9_40582-43}}
\end{figure}

We can identify the closure of $\beta$ as $o9\_40582$.
From Table \ref{table:braid}, $o9\_40582$ is
\[
\begin{split}
\gamma&=[1, 2, 2, 3, 2, 2, 3, 4, 3, 2, 1, 2, 3, 2, 4, 4, 3, 3, 2, 1, 3, 3, 3, 2, 2, 3, 4, 3, 2, 2, 1, 2, \\
&\qquad  2, 3, 2, 3].
\end{split}
\]
Set $\alpha=[2,-3,-2,-1,-3,-3,-2]$.
Then 
\begin{equation}\label{eq:o9_40582-43}
\alpha^{-1}\gamma \alpha=[(1,2,3,4)^5,(1,2,3)^4,2,1,3,3,4,-3]=\beta.
\end{equation}
Hence the closure of $\beta$ is $o9\_40582$.
\end{proof}

\begin{theorem}
The diagram $L12n1638(-\frac{2}{3},1,-2)$ represents the double branched cover of $K13n2958$.
\end{theorem}

\begin{proof}
The diagram $L12n1638(-\frac{2}{3},1,-2)$ is different from Subsection \ref{subsec:t12533-38} only in the third coefficient.
Hence the process goes in the same way as in the proof of Theorem \ref{lem:t12533-38}.
Just replace the tangle $B$ with $-9$ in Fig.~\ref{fig:L14n24287}.

We identify the link with the mirror of $K13n2958$.
We can start from the 5th diagram in Fig.~\ref{fig:L12n789} with modifying the number of twists in the box.
Figure \ref{fig:K13n2958} shows a series of deformations.
\end{proof}

\begin{figure}[th]
\centerline{\includegraphics[bb=0 0 463 360, width=12cm]{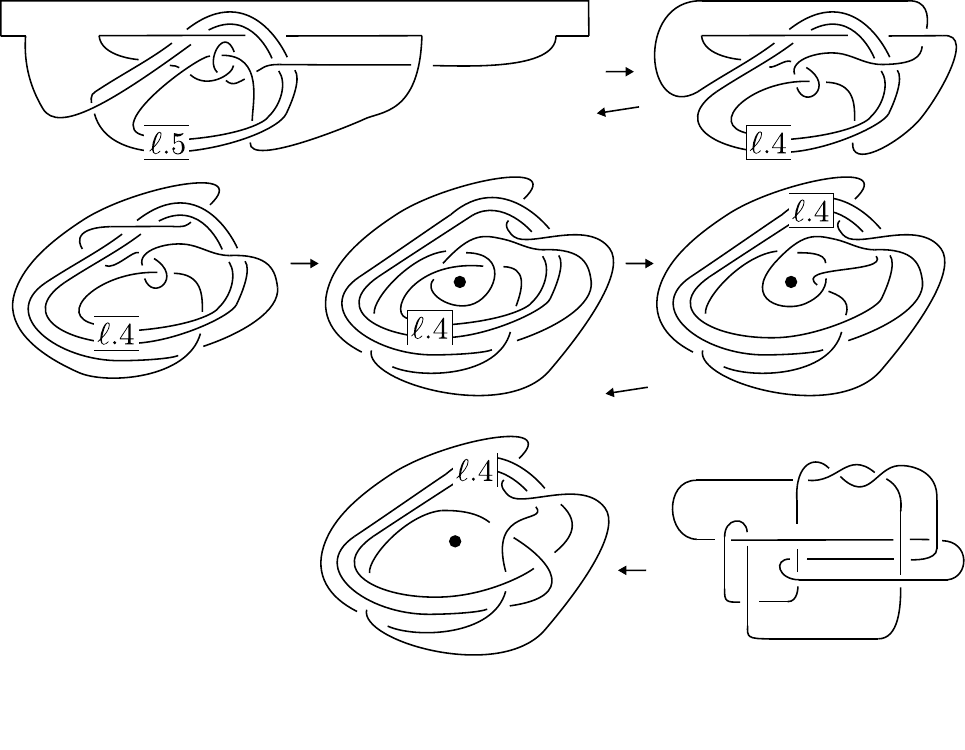}}
\vspace*{8pt}
\caption{Toward the mirror of $K13n2958$ (Bottom Right).
\label{fig:K13n2958}}
\end{figure}

\subsection{$44$--surgery}

\begin{lemma}
$L14n58444(-1,-2,-1)$ represents $44$--surgery on $o9\_40582$.
\end{lemma}

\begin{proof}
Do $+1$ twist on $C_0$, and then $+1$ twist on $C_1$.
Then $C_2$ yields a knot with coefficient $44$, which is in a closed braid form
of $6$ strands (Fig.~\ref{fig:o9_40582-44}).

\begin{figure}[th]
\centerline{\includegraphics[bb=0 0 341 187, width=9cm]{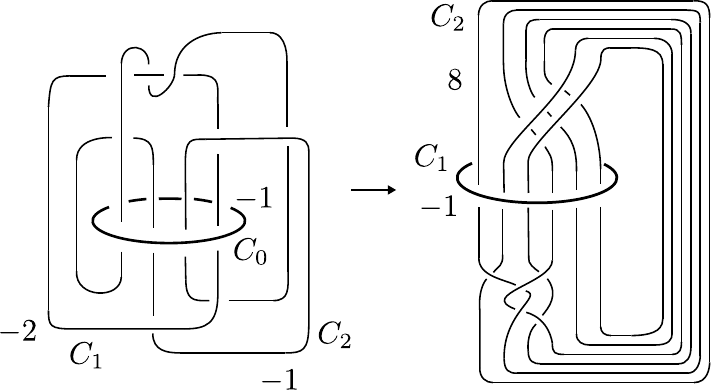}}
\vspace*{8pt}
\caption{Do $+1$ twist on $C_0$, and then $+1$ twist on $C_1$.
\label{fig:o9_40582-44}}
\end{figure}

This braid can be reduced to a $5$--braid $\beta$ as shown in Fig.~\ref{fig:o9_40582-44-2}, where
\[
\beta=[(1,2,3,4)^6,4,3,2,1,2,1,1,-2,3,2,4,3,1,2].
\]

\begin{figure}[th]
\centerline{\includegraphics[bb=0 0 425 184, width=12cm]{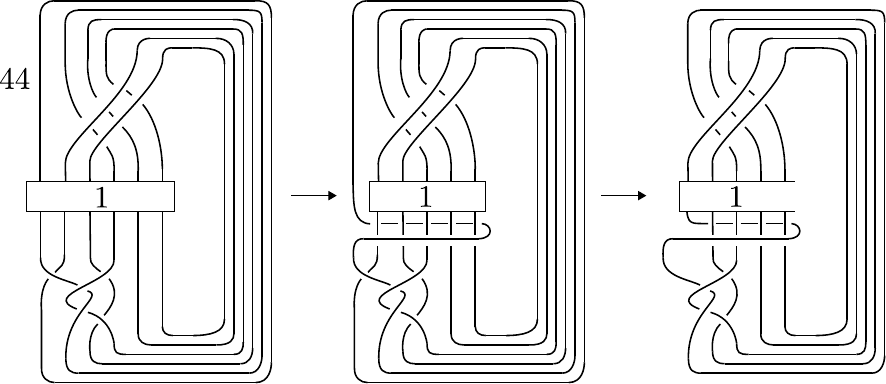}}
\vspace*{8pt}
\caption{The knot can be deformed into the closure of a $5$--braid.
\label{fig:o9_40582-44-2}}
\end{figure}

Set $\alpha=[1,-3,-4]$.
Then
\[
\alpha^{-1}\beta \alpha=[(1,2,3,4)^5,(2,3,4)^4,-3,2,3,4,1,2].
\]
By renaming the generators $\sigma_i$ with $\sigma_{5-i}$, this is
\[
[(4,3,2,1)^5,(3,2,1)^4,-2,3,2,1,4,3]=[(4,3,2,1)^5,(3,2,1)^4,3,2,1,-3,4,3].
\]
Further, taking a conjugation with $[3,2]$ yields
\[
[(4,3,2,1)^5,(3,2,1)^4,1,-3,4,3,3,2]=[(4,3,2,1)^5,(3,2,1)^4,-3,4,3,3,1,2].
\]
Finally, conjugating with $[(3,2,1)^4]$ gives
\[
[(4,3,2,1)^5,-3,4,3,3,1,2,(3,2,1)^4].
\]
If we reverse the order of words, then
this is equal to (\ref{eq:o9_40582-43}).
\end{proof}

\begin{theorem}
The diagram $L14n58444(-1,-2,-1)$ represents the double branched cover of $L13n4413$.
\end{theorem}

\begin{proof}
As in the proof of Theorem \ref{lem:t12533-37},
performing $+1$ twist on $C_2$ gives a strongly invertible diagram as shown in Fig.~\ref{fig:v3437} with
coefficients $(8,7)$ for $(C_0,C_1)$.
Hence the quotient link is obtained from Fig.~\ref{fig:v3437} by replacing
the tangles $A$ and $B$ with $3$ and $4$ respectively.

We need to identity that this is the mirror of  $L13n4413$.
Unlike in Fig.~\ref{fig:K12n407}, it is easier to see how the knotted component, which is the knot $5_2$,
wraps around the  unknotted component (see Fig.~\ref{fig:L13n4413}).
\end{proof}

\begin{figure}[th]
\centerline{\includegraphics[bb=0 0 578 258, width=12cm]{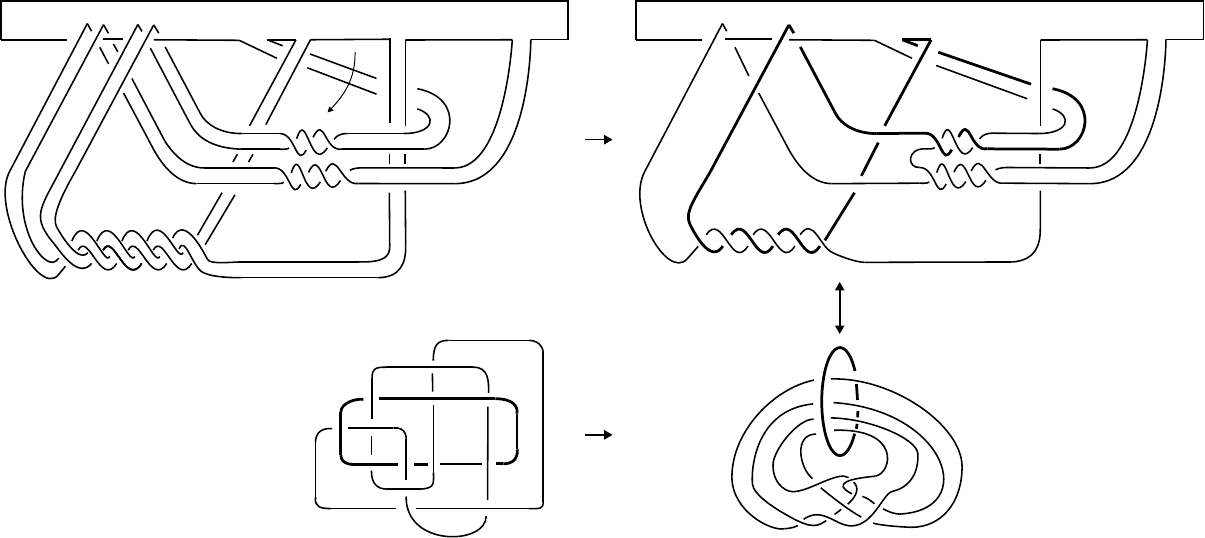}}
\vspace*{8pt}
\caption{An equivalence with the mirror of $L13n4413$ (Bottom left).
\label{fig:L13n4413}}
\end{figure}

\section{$o9\_42675$}

\subsection{$46$--surgery}

\begin{lemma}
$L12n1625(-\frac{1}{4},-3,-1)$ represents $(-46)$--surgery on the mirror of $o9\_42675$.
\end{lemma}

\begin{proof}
Perform $4$ twists on $C_0$ to erase it.
Then, $-1$ twist on $C_1$ changes $C_2$ to a knot with coefficient $-46$, which
is in a closed braid form of $7$ strands as shown in Fig.~\ref{fig:o9_42675-46}.

\begin{figure}[th]
\centerline{\includegraphics[bb=0 0 343 261, width=8cm]{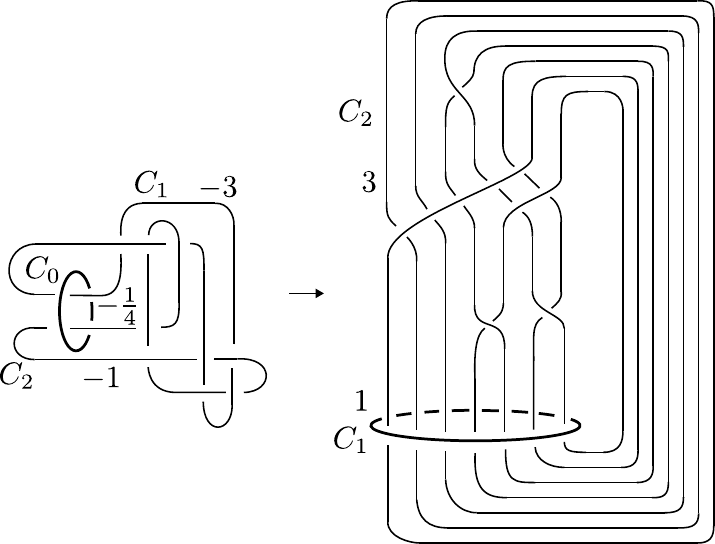}}
\vspace*{8pt}
\caption{Do $4$ twists on $C_0$.
Then $-1$ twist on $C_1$ changes $C_2$ to a knot with coefficient $-46$.
\label{fig:o9_42675-46}}
\end{figure}

As shown in Fig.~\ref{fig:o9_42675-46-2}, this knot is deformed into the closure of a $5$--braid.
Let
\[
\beta=[(1,2,3,4)^5,2,2,3,4,4,3,2,1,4,3,2,1,2,2,2,3].
\]
Then our knot is the mirror of the closure of $\beta$.

\begin{figure}[th]
\centerline{\includegraphics[bb=0 0 561 526, width=12cm]{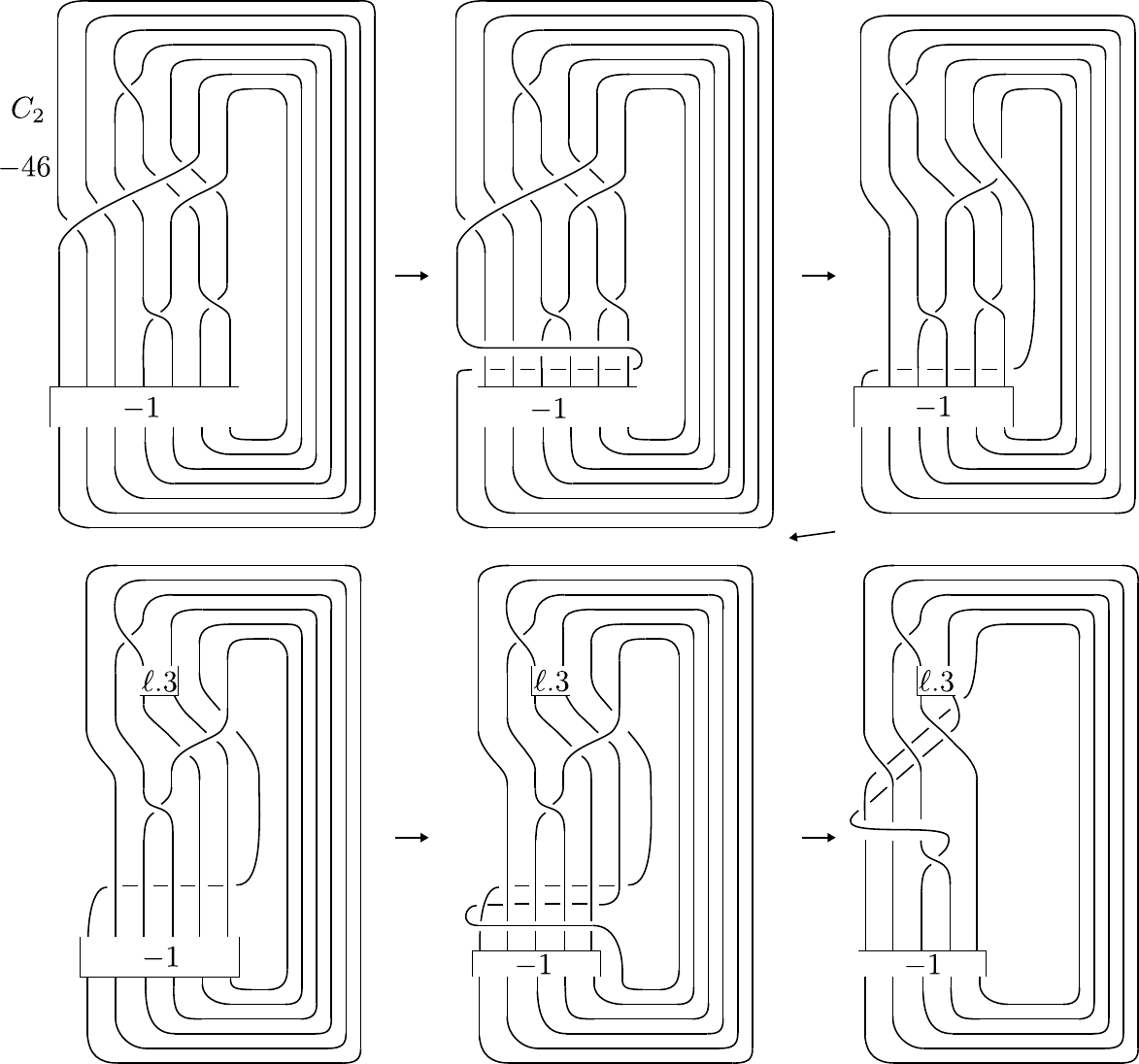}}
\vspace*{8pt}
\caption{The knot is deformed into the closure of a $5$--braid.
\label{fig:o9_42675-46-2}}
\end{figure}

Set $\alpha=[-3,-2,-2,1,-2,-1,3,4,2]$.
Then
\begin{equation}\label{eq:o9_42675-46}
\alpha^{-1}\beta\alpha=[(1,2,3,4)^{10},-4,-3,-2,-2,-3,1,2,-3].
\end{equation}

On the other hand, $o9\_42675$ is
\[
\gamma=[1, 2, 1, 3, 2, 4, 2, 3, 2, 4, 2, 3, 2, 3, 2, 1, 2, 3, 3, 4, 4, 3, 3, 4, 3, 3, 2, 1, 3, 2, 4, 2,  3, 2, 1, 3]
\]
from Table \ref{table:braid}.
Set $\alpha'=[-1,-2,-3,-2,-2]$. Then 
\[
\alpha'^{-1}\gamma \alpha=[(4,3,2,1)^{10},-4,-3,-2,-2,-3,1,2,-3],
\]
which is equal to (\ref{eq:o9_42675-46}).
\end{proof}

\begin{theorem}
The diagram $L12n1625(-\frac{1}{4},-3,-1)$ represents the double branched cover of $L12n702$.
\end{theorem}

\begin{proof}
As in the proof of Theorem \ref{lem:o9_40487-38}, doing $+1$ twist on $C_2$
yields a strongly invertible link $C_1\cup C_2$ with coefficients $(\frac{3}{4},6)$.

In Fig.~\ref{fig:L13n4344}, the tangles $A$ and $B$ are replaced with
$\frac{3}{4}=[0,-1,3]$ and  $-1$, respectively.
Then we can identify that the link is the mirror of $L12n702$ as shown in Fig.~\ref{fig:L12n702}.
\end{proof}

\begin{figure}[th]
\centerline{\includegraphics[bb=0 0 537 122, width=12cm]{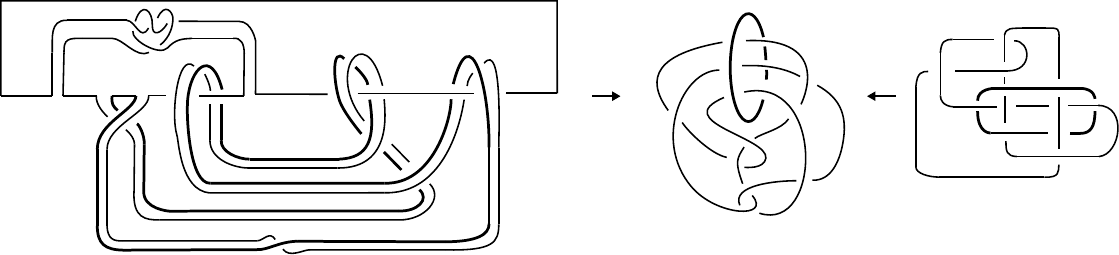}}
\vspace*{8pt}
\caption{An equivalence with the mirror of $L12n702$ (Right).
\label{fig:L12n702}}
\end{figure}

\subsection{$47$--surgery}

Finally, we use $L14n63014$ shown in Fig.~\ref{fig:start}.


\begin{lemma}
$L14n63014(-4,-\frac{3}{2},-1,-1)$ represents
$47$--surgery on $o9\_42675$.
\end{lemma}

\begin{proof}
Perform $+1$ twist on $C_0$, and then $2$ twists on $C_1$.
See Fig.~\ref{fig:o9_42675-47}.
Finally, $+1$ twist on $C_0$ changes $C_2$ into a knot with coefficient $47$ in a closed braid form of $6$ strands.
However, this is deformed into the closure of a $5$--braid as in Fig.~\ref{fig:o9_42675-47-2}.

\begin{figure}[th]
\centerline{\includegraphics[bb=0 0 527 261, width=12cm]{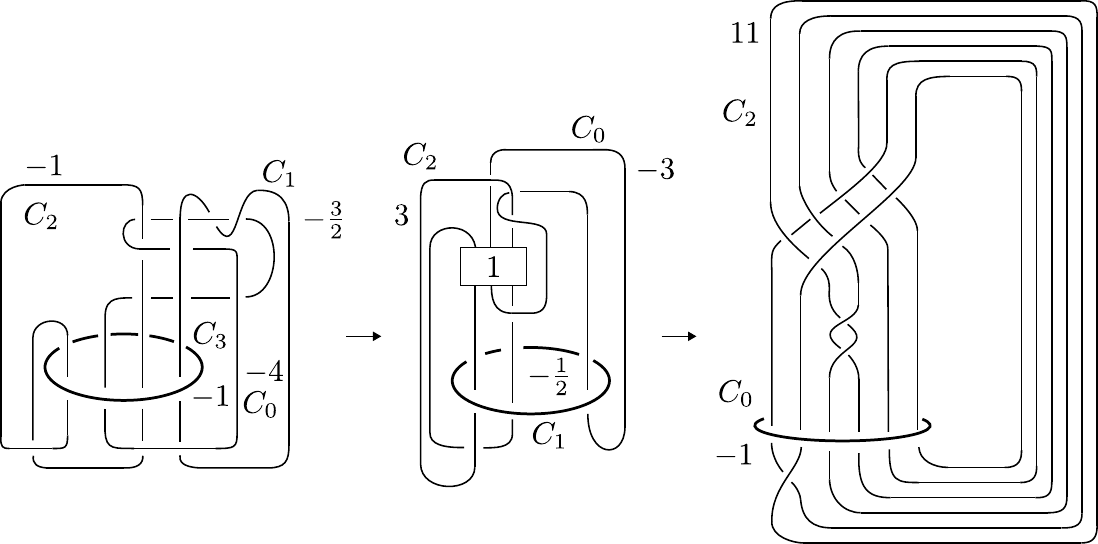}}
\vspace*{8pt}
\caption{Do $+1$ twist on $C_0$, and then $2$ twists on $C_1$.
\label{fig:o9_42675-47}}
\end{figure}

\begin{figure}[th]
\centerline{\includegraphics[bb=0 0 558 261, width=12cm]{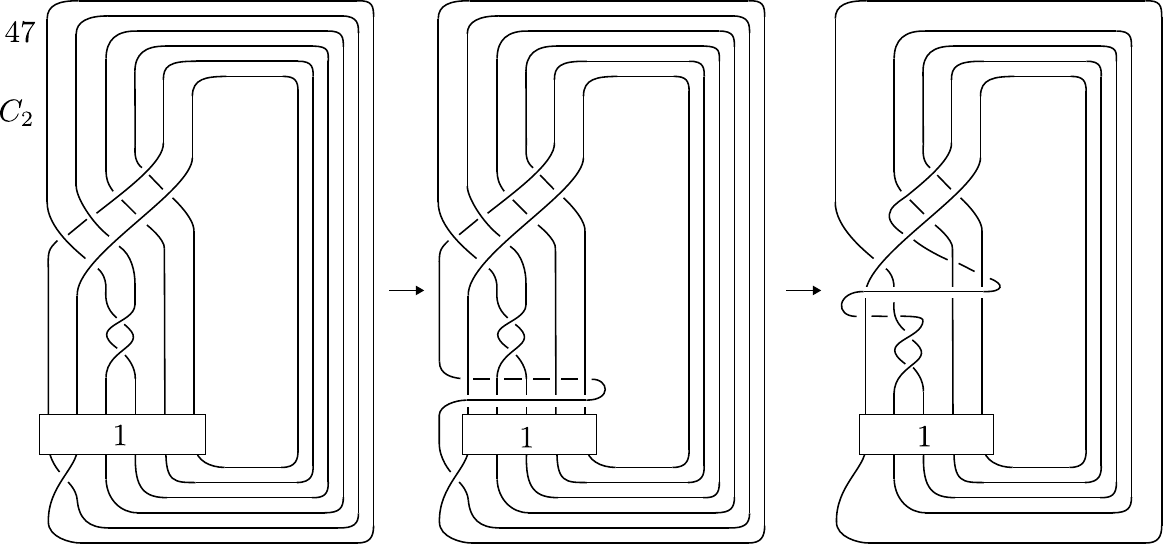}}
\vspace*{8pt}
\caption{The knot is deformed into the closure of a $5$--braid.
\label{fig:o9_42675-47-2}}
\end{figure}

Let 
\[
\beta=[(1,2,3,4)^5,3,3,3,4,4,3,2,1,1,2,4,3,2,1,3,2].
\]
Then our knot is the closure of $\beta$.
Set $\alpha=[-2,-3,-3,-2,-1,3,4,-2,-3,-4]$.  Then
\[
\alpha^{-1}\beta \alpha=[(1,2,3,4)^{10}, -3,2,1,-3,-2,-2,-3,-4].
\]
If we reverse the order of words, this is equal to (\ref{eq:o9_42675-46}).
\end{proof}

\begin{theorem}
The diagram $L14n63014(-4,-\frac{3}{2},-1,-1)$ represents the double branched cover of $K12n730$.
\end{theorem}

\begin{proof}
In the diagram, performing $+1$ twist on $C_2$ yields a strongly invertible diagram.
See Fig.~\ref{fig:L9n24}.
After taking the quotient, the tangle replacement as shown in Fig.~\ref{fig:L9n24} gives
the mirror of the knot $K12n730$.
Figure \ref{fig:K12n730} shows a series of deformations such that
both knots are represented as the closure of the same braid.
\end{proof}

\begin{figure}[th]
\centerline{\includegraphics[bb=0 0 456 134, width=12cm]{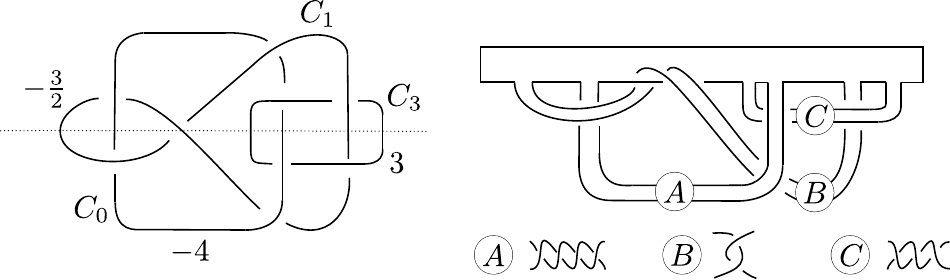}}
\vspace*{8pt}
\caption{Left:  A strongly invertible diagram after $+1$ twist on $C_2$.
Right: The knot after the tangle replacement.
\label{fig:L9n24}}
\end{figure}


\begin{figure}[th]
\centerline{\includegraphics[bb=0 0 528 330, width=12cm]{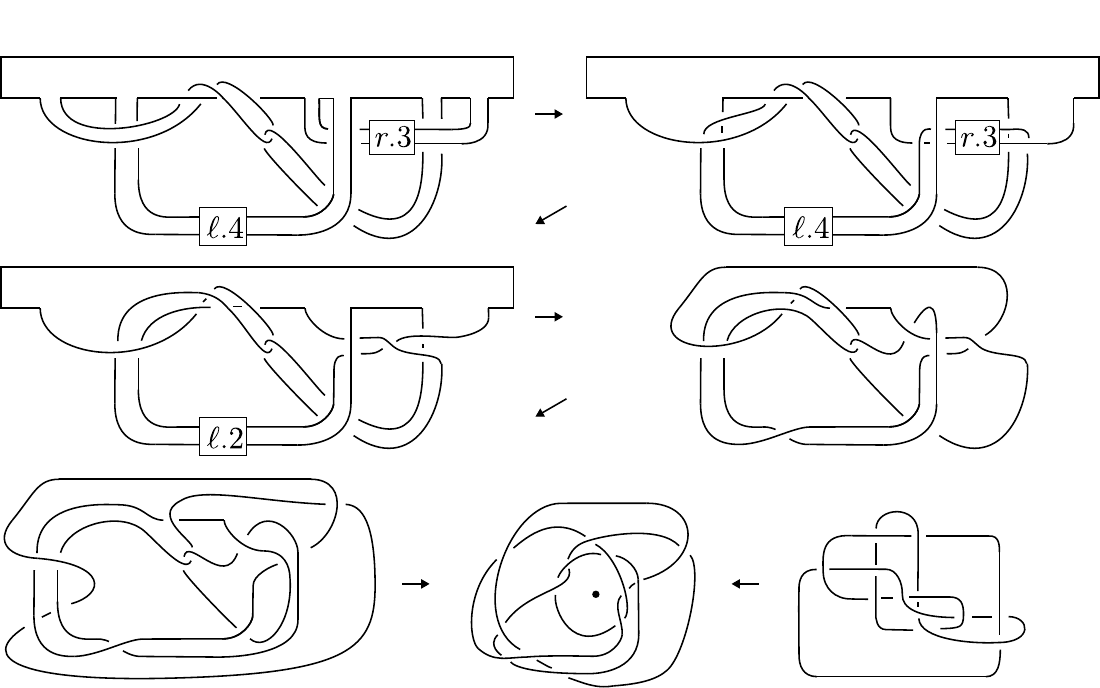}}
\vspace*{8pt}
\caption{The equivalence with the mirror of $K12n730$ (Bottom Right).
A black dot indicates the axis of the closed braid.
\label{fig:K12n730}}
\end{figure}



\begin{thebibliography}{99}

\bibitem{ABG}
C. Anderson, K.~L. Baker, X. Gao, M. Kegel, K. Le, K. Miller, S. Onaran, G. Sangston, S. Tripp, A. Wood and A. Wright,
 L--space knots with tunnel number $>1$ by experiment, 
{\it Exp. Math.}  {\bf 32} (2023), no.~4, 600--614. 

\bibitem{BK}
K.~L. Baker and M. Kegel, 
Census L--space knots are braid positive, except for one that is not, 
{\it Algebr. Geom. Topol.}  {\bf 24} (2024), no.~1, 569--586.

\bibitem{BKM}
K.~L. Baker, M. Kegel and D. McCoy,
Quasi-alternating surgeries,
preprint.
\texttt{arXiv:2409.09839.}


\bibitem{BL}
K.~L. Baker and J. Luecke, Asymmetric L--space knots, 
{\it Geom. Topol.} {\bf 24} (2020), no.~5, 2287--2359. 

\bibitem{CD}
M. Culler, N. Dunfield, M. Goerner and J. Weeks,
SnapPy, a computer program for studying the geometry and topology of 3-manifolds, 
Available at: \texttt{http://snappy.computop.org. }


\bibitem{D1}
N.~M. Dunfield, 
A census of exceptional Dehn fillings, in {\it Characters in low-dimensional topology}, 143--155, Contemp. Math. Centre Rech. Math. Proc., 760 , 
Amer. Math. Soc., RI.

\bibitem{D2}
N.~M. Dunfield, 
Floer homology, group orderability, and taut foliations of hyperbolic 3-manifolds, 
{\it Geom. Topol.} {\bf 24} (2020), no.~4, 2075--2125. 


\bibitem{M}
J.~M. Montesinos-Amilibia, 
Surgery on links and double branched covers of $S\sp{3}$, in {\it Knots, groups, and $3$-manifolds (Papers dedicated to the memory of R. H. Fox)}, pp. 227--259, Ann. of Math. Stud., No. 84, Princeton Univ. Press, Princeton, NJ. 






\end{thebibliography}
\end{document}